\pdfoutput=1

\documentclass{amsart}

\usepackage[margin=1.2 in]{geometry}

\usepackage{verbatim}
\usepackage{xcolor}
\usepackage{tikz}
\usepackage{tikz-cd}
\usepackage{arydshln}
\usepackage{caption}
\usepackage{subcaption}
\usepackage{graphicx}
\usepackage[hyphens]{url}
\usepackage{amsmath}
\usepackage{bbm}
\usepackage{hyperref}
\usepackage{mathrsfs}
\usepackage{gensymb}

\newtheorem{theorem}{Theorem}[section]
\newtheorem{proposition}[theorem]{Proposition}
\newtheorem{lemma}[theorem]{Lemma}

\theoremstyle{definition}

\newtheorem{exercise}[theorem]{Exercise}

\theoremstyle{remark}

\numberwithin{equation}{section}

\newcommand{\mcg}{\mathrm{Mod}_g}

\renewcommand{\tt}{\mathcal{T}_g}

\newcommand{\ml}{\mathcal{ML}_g}

\newcounter{count}

\newcounter{counterk} 
 
\newcounter{counterl} 
 
\newcounter{counterc} 
 
\newcounter{countercc} 
 
\newcounter{countere} 
 
\newcounter{counterd} 
 
\newcounter{countern} 
 
\newcounter{counterr}
 
\begin{document}

\title[Counting problems from the viewpoint of ergodic theory]{Counting problems from the viewpoint of ergodic theory: from primitive integer points to simple closed curves}

\author{Francisco Arana--Herrera}
\address{Institute for Advanced Study, 1 Einstein Drive, Princeton, NJ 08540, USA.}
\email{farana@ias.edu}

\begin{abstract}
	In her thesis, Mirzakhani showed that the number of simple closed geodesics of length $\leq L$ on a closed, connected, oriented hyperbolic surface $X$ of genus $g$ is asymptotic to $L^{6g-6}$ times a constant depending on the geometry of $X$. In this survey we give a detailed account of Mirzakhani’s proof of this result aimed at non-experts. We draw inspiration from classic primitive lattice point counting results in homogeneous dynamics. The focus is on understanding how the general principles that drive the proof in the case of lattices also apply in the setting of hyperbolic surfaces. 
\end{abstract}

\maketitle


\thispagestyle{empty}

\tableofcontents

\section{Introduction}

Let us begin our discussion by recalling a classic result in number theory: the prime number theorem. According to this theorem, the number $\pi(N)$ of positive prime integers $p$ such that $p \leq N$ satisfies the following asymptotic estimate as $N \to \infty$,
\[
\pi(N) \sim \frac{N}{\log(N)},
\]
where the symbol $\sim$ represents the fact that the following identity holds
\[
\lim_{N \to \infty} \frac{\pi(N)}{N / \log (N)} = 1.
\]

Analogous results also hold in the more geometric setting of hyperbolic surfaces.  A closed geodesic on a hyperbolic surface is said to be primitive if it cannot be represented as a concatenation of multiple copies of a shorter closed geodesic. Given a closed, oriented hyperbolic surface $X$ and a parameter $L > 0$, denote by $c(X,L)$ the number of non-oriented, primitive closed geodesics on $X$ of length $\leq L$. The prime geodesic theorem states that, for any closed, connected, oriented hyperbolic surface $X$, the following asymptotic estimate holds as $L \to \infty$,
\[
c(X,L) \sim \frac{e^L}{2L}.
\]
This result shows in particular that the asymptotics of $c(X,L)$ as $L \to \infty$ do not depend on the geometry of $X$ nor on its topology. The first proof of this theorem was given by Huber and Selberg using analytic methods. See \cite[Chapter 9]{Bus92} for a detailed discussion of this proof.

In the more geometric setting of hyperbolic surfaces we can push our curiosity even further. A closed geodesic on a hyperbolic surface is said to be simple if it does not intersect itself. Given a closed, connected, oriented hyperbolic surface $X$ and a parameter $L >0 $, denote by $s(X,L)$ the number of simple closed geodesics on $X$ of length $\leq L$. What are the asymptotics of this quantity as $L \to \infty$? Do they depend on the geometry of $X$? Do they depend on the topology of $X$?

These questions, which a priori seem very similar to the ones answered by the prime geodesic theorem, remained out of reach for quite a long time. The analytic perspective of Huber and Selberg was not of much use in this setting as it could not distinguish simple closed geodesics among primitive ones. Some progress towards answering these questions was made by Birman and Series \cite{BS85}, McShane and Rivin \cite{MR95a,MR95b}, and Rivin \cite{R01}. The first major breaktrough would come through the work of Mirkzakhani, who, in her thesis \cite{Mir04}, proved the following outstanding theorem.

\begin{theorem} \cite[Theorem 1.1]{Mir08b}
	\label{theo:main}
	Let $X$ be a closed, connected, oriented hyperbolic surface of genus $g \geq 2$. Then, there exists $s(X) > 0$ such that the following asymptotic estimate holds as $L \to \infty$,
	\[
	s(X,L) \sim s(X) \cdot L^{6g-6}.
	\]
\end{theorem}

Mirzakhani's proof of Theorem \ref{theo:main} uses ergodic theory in a crucial way. Her proof also makes important use of a couple of other breakthroughs of herself: her famous formulas for the total Weil-Petersson volumes of moduli spaces and her famous integration formulas over moduli space. 

The tools developed by Mirzakhani in her thesis have broad-ranging consequences in many fields of mathematics. Let us highlight the following remarkably concrete consequence of her work.

\begin{theorem}
	\cite[Corollary 1.4]{Mir08b}
	\label{theo:main_2}
	On any closed, oriented hyperbolic surface of genus $2$ it is $48$ times more likely for a random long simple closed geodesic to be non-separating rather than separating.
\end{theorem}

The main goal of this survey is to give a detailed account of Mirzakhani’s proof of Theorem \ref{theo:main} aimed at non-experts. We will draw inspiration from classic primitive lattice point counting results in homogeneous dynamics. Although we will cover the necessary background, the focus will be on understanding how the general principles that drive the proof in the case of lattices also apply in the setting of hyperbolic surfaces. In particular, we will take for granted several fundamental results about Teichmüller spaces and mapping class groups and focus on understanding their applications.

\subsection*{Organization of this survey.} In \S2 we study counting problems for primitive lattice points in the Euclidean plane. This discussion will later guide the proof of Theorem \ref{theo:main}. In \S3 we cover the background material on hyperbolic surfaces, Teichmüller spaces, and simple closed curves needed to understand the proof of Theorem \ref{theo:main}. In \S4 we discuss Mirzakhani's famous formulas for the total Weil-Petersson volumes of moduli spaces and her famous integration formulas over moduli space. In \S5 we give a complete proof of Theorem \ref{theo:main}. In \S 6 we give a brief overview of several counting results for closed curves on surfaces and other related objects that have been proved since the debut of Mirkzakani's thesis.

\subsection*{Other surveys.} For a review of similar topics from a more algebro-geometric perspective see \cite{Wol13}. For a survey covering the full range of Mirzakhani's outstanding research beyond the few results discussed here see \cite{Wri19}. For discussions of geodesic counting theorems from the point of view geodesic currents see \cite{EU18} and \cite{ES20}.

\subsection*{Acknowledgments.} The author is very grateful to Alex Wright and Steve Kerckhoff for their invaluable advice, patience, and encouragement. The author would also like to thank Alex Wright, Anton Zorich, and Howard Masur for their helpful comments on an earlier version of this survey. This survey got started as a set of notes for a minicourse taught by the author at the CMI-HIMR Dynamics and Geometry Online Summer School. The author is very grateful to Viveka Erlandsson, John Mackay, and Jens Marklof for giving him the chance to teach this minicourse. This survey was finished while the author was a member of the Institute for Advanced Study (IAS). The author is very grateful to the IAS for its hospitality. This material is based upon work supported by the National Science Foundation under Grant No. DMS-1926686.

\section{Counting primitive lattice points in the Euclidean plane}

\subsection*{Outline of this section.} In this section we study counting problems for primitive lattice points in the Euclidean plane. The techniques introduced in this section will serve as a rough guide for the proof of Theorem \ref{theo:main} that will be discussed in \S 5. 

\subsection*{Counting primitive integer points.} Consider the integer lattice $\mathbf{Z}^2 \subseteq \mathbf{R}^2$. In analogy with the definition of prime numbers, a vector $v \in \mathbf{Z}^2$ is said to be primitive if it cannot be written as a non-negative integer multiple of another vector in $\mathbf{Z}^2$. Equivalently, a vector $v = (a,b) \in \mathbf{Z}^2$ is said to be primitive if the greatest common divisor of $a$ and $b$ is $1$. Denote by $\smash{\mathbf{Z}^2_\mathrm{prim}} \subseteq \mathbf{Z}^2$ the subset of all primitive vectors of $\mathbf{Z}^2$. See Figure \ref{fig:primitive_lattice}. Notice $\smash{\mathbf{Z}^2_\mathrm{prim}} \subseteq \mathbf{Z}^2$ is not a sublattice, not even a subgroup. Denote by $\|\cdot \|$ the Euclidean norm on $\mathbf{R}^2$. For every $L >0$ consider the counting function
\begin{equation}
\label{eq:prim}
p\left(\mathbf{Z}^2,L \right) := \# \{v \in \mathbf{Z}^2_\mathrm{prim}  \colon \| v \| \leq L \},
\end{equation}

\begin{figure}[h!]
	\centering
	\includegraphics[width=.27\textwidth]{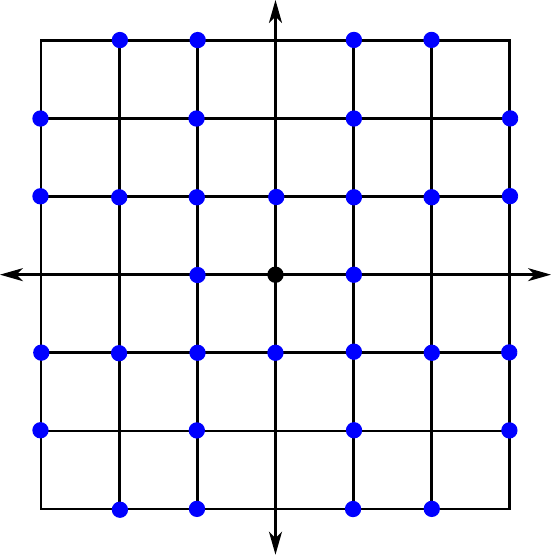}
	\vspace{+.2cm}
	\caption{Primitive vectors of the integer lattice $\mathbf{Z}^2$.} \label{fig:primitive_lattice} 
\end{figure}

Just as in the case of the counting function $\pi(N)$ introduced in \S 1, we are interested in the asymptotics of $p(\smash{\mathbf{Z}^2},L)$ as $L \to \infty$. The main goal of this section is to discuss a proof of the following asymptotic estimate for the counting function $p(\smash{\mathbf{Z}^2},L)$.

\begin{theorem}
	\label{theo:main_lec_1}
	The following asymptotic estimate holds as $L \to \infty$,
	\[
	p\left(\mathbf{Z}^2,L\right) \sim \frac{6}{\pi} \cdot L^2.
	\]
\end{theorem}

To prove Theorem \ref{theo:main_lec_1} we begin by rewriting the counting function $p(\mathbf{Z}^2,L)$ in a more convenient way. More concretely, notice that for every $L > 0$,
\begin{align*}
	p(\smash{\mathbf{Z}^2},L) &= \# \{v \in \smash{\mathbf{Z}^2_\mathrm{prim}} \colon \| v \| \leq L \} \\
	&= \# \left\lbrace v \in \smash{\mathbf{Z}^2_\mathrm{prim}}  \colon \| \textstyle\frac{1}{L} \cdot v \| \leq 1 \right\rbrace \\
	&= \# \left\lbrace v \in \textstyle\frac{1}{L} \cdot \smash{\mathbf{Z}^2_\mathrm{prim}}  \colon \| v \| \leq 1 \right\rbrace.
\end{align*}
It is natural then to study the rescaled lattice $\smash{\frac{1}{L} \cdot \mathbf{Z}^2_\mathrm{prim}} \subseteq \mathbf{R}^2$ as $L \to \infty$. We do this by means of a measure theoretic approach. For every $L > 0$ consider the counting measure on $\mathbf{R}^2$ given by
\[
\nu_L^\mathrm{prim}  := \frac{1}{L^2} \cdot \sum_{v \in \mathbf{Z}_\mathrm{prim}^2} \delta_{\frac{1}{L} \cdot v}.
\]
Notice that, if $B \subseteq \mathbf{R}^2$ denotes the unit ball centered at the origin, then
\[
\nu_L^\mathrm{prim}(B) = \frac{p\left(\smash{\mathbf{Z}^2},L\right)}{L^2}.
\]
This setup reduces the original problem of proving an asymptotic estimate for the counting function $p\left(\smash{\mathbf{Z}^2},L\right)$  to the problem of understanding the asymptotic behaviour of the measures $\nu_L^\mathrm{prim}$ as $L \to \infty$.

\subsection*{The Lebesgue measure.} To develop some intuition, let us first consider a much simpler family of counting measures on $\mathbf{R}^2$. For every $L > 0$ consider the counting measure $\nu_L$ on $\mathbf{R}^2$ given by
\begin{equation}
\label{eq:leb_count}
\nu_L  := \frac{1}{L^2} \cdot \sum_{v \in \mathbf{Z}^2} \delta_{\frac{1}{L} \cdot v}.
\end{equation}
Denote by $\nu$ the standard Lebesgue measure on $\mathbf{R}^2$. Based solely on geometric intuition, one expect that the following weak-$\star$ convergence of measures holds,
\begin{equation}
\label{eq:lim_1}
\lim_{L \to \infty} \nu_L  = \nu.
\end{equation}
More concretely, one expects that for every continuous, compactly supported function $f \colon \mathbf{R}^2 \to \mathbf{R}$, 
\begin{equation}
\label{eq:weak_def}
\lim_{L \to \infty}  \int_{\mathbf{R}^2} f \thinspace d\nu_L = \int_{\mathbf{R}^2} f \thinspace d \nu.
\end{equation}
The proof of (\ref{eq:lim_1}) will be left as an exercise towards the end of this section. See Exercise \ref{ex:lebesgue_lim}. Let us point out for the moment that the main idea behind the proof of (\ref{eq:lim_1}) is the fact that every weak-$\star$ limit point of the sequence of counting measures $(\nu_L)_{L > 0}$ is translation invariant. 

\begin{exercise}
	\label{ex:finite_index}
	What would happen with (\ref{eq:lim_1}) if we considered a finite index subgroup of $\mathbf{Z}^2$ instead of all $\mathbf{Z}^2$ in the definition of the counting measures $(\nu_L)_{L>0}$ in (\ref{eq:leb_count})?
\end{exercise}

\subsection*{Invariance of counting measures.} The discussion above suggests that to prove an analogue of (\ref{eq:lim_1}) for the counting measures $(\smash{\nu_L^\mathrm{prim}})_{L>0}$, one should try to understand the behavior of the weak-$\star$ limit points of this sequence. A priori, one has no reason to expect the limit points of this sequence will be translation invariant, as we will conclude a fortiori. Nevertheless, there is still a relevant notion of invariance present in this setting. Indeed, consider the discrete matrix group
\[
\mathrm{SL}(2,\mathbf{Z}) := \left\lbrace \left( \begin{array}{c c}
a & b \\ c & d
\end{array}\right) \in \mathrm{Mat}_{2\times2}(\mathbf{R})\ \bigg\vert \ a,b,c,d \in \mathbf{Z}, ab-cd = 1 \right\rbrace
\]
acting on $\mathbf{R}^2$ by linear transformations. The following crucial exercise will provide the notion of invariance we will use to study the counting measures $(\smash{\nu_L^\mathrm{prim}})_{L>0}$.

\begin{exercise}
	\label{ex:prim_orb}
	Show that the $\mathrm{SL}(2,\mathbf{Z})$ orbit of the vector $(1,0) \in \mathbf{R}^2$ is precisely $\mathbf{Z}^2_{\mathrm{prim}} \subseteq \mathbf{Z}^2$. Conclude that any weak-$\star$ limit point of the sequence $(\smash{\nu_L^\mathrm{prim}})_{L>0}$ is $\mathrm{SL}(2,\mathbf{Z})$-invariant. \textit{Hint: Use Bézout's identity for greatest common divisors.} 
\end{exercise}

Exercise \ref{ex:prim_orb} ensures that any weak-$\star$ limit point of the sequence $(\nu_L^\mathrm{prim})_{L>0}$ is $\mathrm{SL}(2,\mathbf{Z})$-invariant. Using ergodic theory we will show this property greatly constraints the possible weak-$\star$ limit points.

\subsection*{Ergodic theory.} Let $(X,\mathcal{A})$ be a measurable space and $G$ be a countable group acting on $(X,\mathcal{A})$ by measure preserving transformations. We say a $\sigma$-finite measure $\mu$ on $(X,\mathcal{A})$ is $G$-invariant if $\mu(g.A) = \mu(A)$ for every measurable subset $A \in \mathcal{A}$ and every $g \in G$. Furthermore, we say a $G$-invariant measure $\mu$ on $(X,\mathcal{A})$ is $G$-ergodic if it admits no non-trivial $G$-invariant subsets, that is, if $\mu(A) = 0$ or $\mu(X\setminus A) = 0$ for every $A \in \mathcal{A}$ such that $g.A = A$ for every $g \in G$. Equivalently, a $G$-invariant measure $\mu$ on $(X,\mathcal{A})$ is ergodic if every $G$-invariant measurable function is constant almost everywhere with respect to $\mu$. The following exercise characterizes $G$-invariant measures on $(X,\mathcal{A})$ that are absolutely continuous with respect to a $G$-ergodic measure.

\begin{exercise}
	\label{ex:ergodic}
	Let $(X,\mathcal{A})$ be a measurable space and $G$ be a countable group acting on $(X,\mathcal{A})$ by measure preserving transformations. Suppose that $\nu$ is a $G$-ergodic measure on $(X,\mathcal{A})$ and that $\mu$ is a $G$-invariant measure on $(X,\mathcal{A})$ that is absolutely continuous with respect to $\mu$. Show that $\mu$ is a non-negative constant multiple of $\nu$. \textit{Hint: Show that the Radon-Nikodym derivative of $\mu$ with respect to $\nu$ is $G$-invariant and use the $G$-ergodicity of $\nu$ to show this derivative is constant.}
\end{exercise}

In the setting of counting problems for primitive lattice points, the following ergodicity result will be relevant for us. Although we will not prove it in this survey, let us at least mention that this result can be proved using the ergodicity of the horocycle flow on the unit tangent of the modular curve $T^1\mathcal{M}_1 := \mathrm{SL}(2,\mathbf{Z}) \backslash \mathrm{SL}(2,\mathbf{R})$. 

\begin{theorem}
	\label{theo:sl_ergodic}
	The Lebesgue measure $\nu$ on $\mathbf{R}^2$ is ergodic with respect to the linear action of $\mathrm{SL}(2,\mathbf{Z})$.
\end{theorem}

\subsection*{Portmanteau's theorem.} We now state a classic theorem of Portmanteau that gives useful alternative characterizations to the definition of weak-$\star$ convergence in (\ref{eq:weak_def}).

\begin{theorem}
	\label{theo:port_1}
	Let $X$ be a metric space and $(\mu_L)_{L>0}$ be a sequence of locally finite Borel measures on $X$ converging in the weak-$\star$ topology to a Borel measure $\mu$ on $X$. Then, for every open subset $U \subseteq X$,
	\[
	\mu(U) \leq \liminf_{L \to \infty} \mu_L(U).
	\]
	Additionally, for every compact subset $K \subseteq X$ such that $\mu(\partial K) = 0$,
	\begin{equation}
	\label{eq:port_2}
	\lim_{L \to \infty} \mu_L(K) = \mu(K).
	\end{equation}
\end{theorem}

\begin{exercise}
	Asuming the weak-$\star$ convergence in (\ref{eq:lim_1}) holds, find counterexamples to identity (\ref{eq:port_2}) in Theorem \ref{theo:port_1} when either $K \subseteq \mathbf{R}^2$ is not compact or does not satisfy $\nu(\partial K) = 0$.
\end{exercise}

\subsection*{Limit points of counting measures.} Using Exercises \ref{ex:prim_orb} and \ref{ex:ergodic}, and Theorem \ref{theo:sl_ergodic} and \ref{theo:port_1}, one can greatly restrict the possible weak-$\star$ limit points of the  the sequence of counting measures $(\smash{\nu_L^\mathrm{prim}})_{L>0}$.

\begin{proposition}
	\label{prop:prim_meas_1}
	Every weak-$\star$ limit point $\nu^\mathrm{prim}$ of the sequence of counting measures $(\smash{\nu_L^\mathrm{prim}})_{L>0}$ is of the form $\nu^\mathrm{prim} = c \cdot \nu$ for some constant $c\geq 0$.
\end{proposition}

\begin{proof}
	Let $\smash{\nu^\mathrm{prim}}$ be a weak-$\star$ limit point of the sequence $\smash{(\nu_L^\mathrm{prim})_{L>0}}$. Exercise \ref{ex:prim_orb} guarantees that $\smash{\nu^\mathrm{prim}}$ is $\mathrm{SL}(2,\mathbf{Z})$-invariant. Theorem \ref{theo:sl_ergodic} ensures that the Lebesgue measure $\nu$ on $\mathbf{R}^2$ is ergodic with respect to the linear action of $\mathrm{SL}(2,\mathbf{Z})$. Thus, by Exercise \ref{ex:ergodic}, to prove $\smash{\nu^\mathrm{prim}} = c \cdot \nu$ for some constant $c \geq 0$, it is enough to check that $\smash{\nu^\mathrm{prim}}$ is absolutely continuous with respect to $\nu$. 
	
	To prove this let us consider an arbitrary Borel measurable subset $A \subseteq \mathbf{R}^2$ such that $\nu(A) = 0$. Our goal is to show that $\nu^\mathrm{prim}(A) = 0$. Let $\delta > 0$ be arbitrary. The outer regularity of the Lebesgue measure $\nu$ guarantees that one can find a countable collection of open squares $\{B_i\}_{i\in \mathbf{N}}$ such that
	\begin{equation}
	\label{eq:A0}
	A \subseteq \bigcup_{i \in \mathbf{N}} B_i, \quad \sum_{i \in \mathbf{N}} \nu(B_i) \leq \delta.
	\end{equation}
	
	To bound $\nu^\mathrm{prim}(A)$ let us first give a rough bound of the measure $\nu^\mathrm{prim}(B)$ of an arbitrary open square $B \subseteq \mathbf{R}^2$. Denote by $\epsilon > 0$ the side length of $B$. Notice that, for every $L > 0$, 
	\begin{equation}
	\label{eq:A1}
	\nu^\mathrm{prim}_L(B) \leq \nu_L(B) = \frac{\#\left(\mathbf{Z}^2 \cap (L \cdot B)\right)}{L^2}.
	\end{equation}
	The set $L \cdot B \subseteq \mathbf{R}^2$ is an open square of side length $L \cdot \epsilon$. Suppose $L > 0$ is large enough so that this side length satisfies $L \cdot \epsilon \geq 1$. Consider the set $S \subseteq \mathbf{R}^2$ obtained by taking the union of disjoint open squares of side length $1/2$ centered at every point $v \in \mathbf{Z}^2 \cap (L \cdot B)$. This set is contained in an open square $B' \subseteq \mathbf{R}^2$ of side length $L \cdot \epsilon + 1$. It follows that 
	\begin{equation}
	\label{eq:A2}
	\#\left(\mathbf{Z}^2 \cap (L \cdot B)\right) \cdot (1/4) = \nu(S) \leq \nu(B') = (L \cdot \epsilon + 1)^2 \leq 4 \cdot L^2 \cdot \epsilon^2 = 4 \cdot L^2 \cdot \nu(B).
	\end{equation}
	Putting together (\ref{eq:A1}) and (\ref{eq:A2}) we deduce
	\begin{equation}
	\label{eq:A3}
	\limsup_{L \to \infty} \nu_L^\mathrm{prim}(B) \leq 16 \cdot \nu(B).
	\end{equation}
	
	We now bound $\nu^\mathrm{prim}(A)$. Using the cover in (\ref{eq:A0}) and the subadditivity of $\nu^\mathrm{prim}$ we deduce
	\begin{equation}
	\label{eq:a1}
	\nu^\mathrm{prim}(A) \leq \sum_{i \in \mathbf{N}} \nu^\mathrm{prim}(B_i).
	\end{equation}
	Theorem \ref{theo:port_1} ensures that, for every $i \in \mathbf{N}$,
	\begin{equation}
	\label{eq:a2}
	\nu^\mathrm{prim}(B_i) \leq \liminf_{L \to \infty} \nu_L^\mathrm{prim}(B_i)  \leq \limsup_{L \to \infty} \nu_L^\mathrm{prim}(B_i).
	\end{equation}
	The identity in (\ref{eq:A3}) guarantees that, for every $i \in \mathbf{N}$,
	\begin{equation}
	\label{eq:a3}
	\limsup_{L \to \infty} \nu_L^\mathrm{prim}(B_i) \leq 16 \cdot \nu(B_i).
	\end{equation}
	Putting together (\ref{eq:a1}), (\ref{eq:a2}), and (\ref{eq:a3}), and using the inequality in (\ref{eq:A0}) we deduce
	\[
	\nu^\mathrm{prim}(A) \leq \sum_{i \in \mathbf{N}} 16 \cdot \nu(B_i) \leq 16\cdot \delta.
	\]
	As $\delta > 0$ is arbitrary we conclude $\nu^\mathrm{prim}(A) = 0$, thus finishing the proof.
\end{proof}

\subsection*{The space of unimodular lattices.} Our next goal is to show that the constant $c \geq 0$ in the conclusion of Proposition \ref{prop:prim_meas_1} is positive and independent of the limit point $\nu^\mathrm{prim}$. We do this by considering an averaging argument over the space of unimodular lattices of $\mathbf{R}^2$ up to rotation.

A lattice $\Lambda \subseteq \mathbf{R}^2$ is the $\mathbf{Z}$-span of an $\mathbf{R}$-basis of $\mathbf{R}^2$. A marking of a lattice $\Lambda \subseteq \mathbf{R}^2$ is a choice of positively oriented $\mathbf{R}$-basis $(v_1,v_2)$ of $\mathbf{R}^2$ such that $\Lambda = \mathrm{span}_{\mathbf{Z}}(v_1,v_2)$. The covolume of a lattice $\Lambda := \mathrm{span}_{\mathbf{Z}}(v_1,v_2) \subseteq \mathbf{R}^2$ is defined as $\mathrm{covol}(\Lambda) := |\mathrm{det}(v_1,v_2)|$. This definition is independent of the choice of marking $(v_1,v_2)$. A lattice is said to be unimodular if it has unit covolume. The group
\[
\mathrm{SL}(2,\mathbf{R}) := \left\lbrace \left( \begin{array}{c c}
a & b \\ c & d
\end{array}\right) \in \mathrm{Mat}_{2\times2}(\mathbf{R})\ \bigg\vert \ a,b,c,d \in \mathbf{R}, ab-cd = 1 \right\rbrace
\]
acts transitively on the set of unimodular lattices of $\mathbf{R}^2$. The stabilizer of the integer lattice $\mathbf{Z}^2 \subseteq \mathbf{R}^2$ is the group $\mathrm{SL}(2,\mathbf{Z}) \subseteq \mathrm{SL}(2,\mathbf{R})$. We can thus identify the space of unimodular lattices of $\mathbf{R}^2$ with the quotient $T^1 \mathcal{M}_1 := \mathrm{SL}(2,\mathbf{R})/\mathrm{SL}(2,\mathbf{Z})$. The corresponding identification maps the equivalence class of matrices $[A] \in T^1 \mathcal{M}_1$ to the lattice $A \cdot \mathbf{Z}^2 \subseteq \mathbf{R}^2$. The orthogonal group
\[
\mathrm{SO}(2,\mathbf{R}) := \left\lbrace
\left(\begin{array}{c c}
\cos \theta & - \sin \theta \\
\sin \theta & \cos \theta
\end{array} \right) \in \mathrm{Mat}_{2 \times 2}(\mathbf{R})
\ \bigg\vert \ \theta \in [0,2\pi] \right \rbrace
\]
acts on the space of unimodular lattices through its linear action on $\mathbf{R}^2$, or, equivalently, by left multiplication on $T^1 \mathcal{M}_1$. The corresponding quotient $\mathcal{M}_1:= \mathrm{SO}(2,\mathbf{R})\backslash \mathrm{SL}(2,\mathbf{R})/\mathrm{SL}(2,\mathbf{Z})$ can thus be identified with the space of unimodular lattices of $\mathbf{R}^2$ up to rotation.

It will be convenient to consider the following alternative description of the space $\mathcal{M}_1$. Denote
\[
\mathbf{H}^2 := \{z \in \mathbf{C} \ | \ \Im(z) > 0\}.
\]
The upper half-space $\mathbf{H}^2$ can be identified with the space $\mathrm{SO}(2,\mathbf{R})\backslash \mathrm{SL}(2,\mathbf{R})$ of marked unimodular lattices of $\mathbf{R}^2$ up to rotation via the map which sends $z \in \mathbf{H}^2$ to the marked unimodular lattice $\Lambda(z) := \mathrm{span}_{\mathbf{Z}}(c_z e_1,c_z z) \subseteq \mathbf{R}^2$, where $c_z := \Im(z)^{-1/2} > 0$. The discrete group $\mathrm{SL}(2,\mathbf{Z})$ acts properly discontinuously on $\mathbf{H}^2$ by Möbius transformations in the following way,
\[
\left(\begin{array}{c c}
a & b \\
c & d
\end{array}\right) \cdot z = \frac{az+b}{cz+d}, \quad z \in \mathbf{H}^2.
\]
The identification above intertwines this action with the action of $\mathrm{SL}(2,\mathbf{Z})$ on $\mathrm{SO}(2,\mathbf{R})\backslash \mathrm{SL}(2,\mathbf{R})$ by right multiplication. Thus, the quotient $\mathcal{M}_1 := \mathbf{H}^2/\mathrm{SL}(2,\mathbf{Z}^2)$, commonly known as the modular curve, can be identified with the space of unimodular lattices of $\mathbf{R}^2$ up to rotation. See Figure \ref{fig:mod} for a representation of this space using a fundamental domain of the action of $\mathrm{SL}(2,\mathbf{Z})$ on $\mathbf{H}^2$.

Consider the Lebesgue class measure $\mu$ on $\mathbf{H}^2$ given in the coordinates $z  = x + iy$ by
\[
\mu = \frac{dx \thinspace dy}{y^2}.
\]
A direct computation shows that $\mu$ is preserved by the action of $\mathrm{SL}(2,\mathbf{Z})$ on $\mathbf{H}^2$. Denote by $\widehat{\mu}$ the local pushforward of $\mu$ to the quotient $\mathcal{M}^1 := \mathbf{H}^2/\mathrm{SL}(2,\mathbf{Z})$. More explicitely, $\widehat{\mu}$ is the measure on $\mathcal{M}^1$ obtained by restricting $\mu$ to the fundamental domain of the action of $\mathrm{SL}(2,\mathbf{Z})$ on $\mathbf{H}^2$ in Figure \ref{fig:mod}. Local pushforwards of measures will be discussed in more detail in \S3 following a fundamental domain independent perspective. A direct computation shows that $\widehat{\mu}(\mathcal{M}_1) = \pi/3$.

\begin{figure}[h]
	\centering
	\begin{subfigure}[b]{0.4\textwidth}
		\centering
		\includegraphics[width=.85\textwidth]{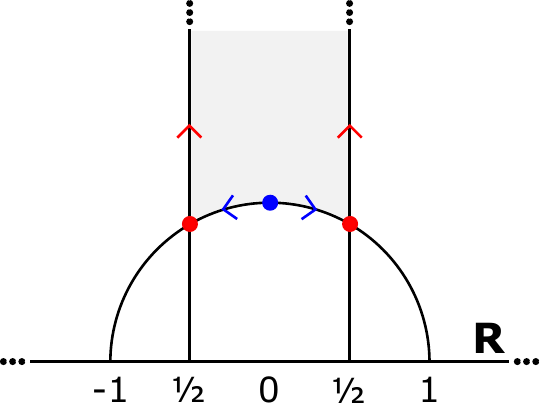}
		\caption{Fundamental domain.}
	\end{subfigure}
	\quad \quad \quad
	~ 
	\begin{subfigure}[b]{0.4\textwidth}
		\centering
		\includegraphics[width=.15\textwidth]{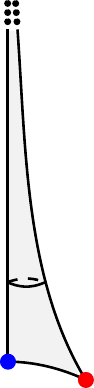}
		\caption{Modular curve $\mathcal{M}_1$.}
	\end{subfigure}
	\caption{A fundamental domain of $\mathrm{SL}(2,\mathbf{Z})$ action on $\mathbf{H}^2$ and the quotient $\mathcal{M}_1$.} 
	\label{fig:mod}
\end{figure}

\subsection*{Averaging over the space of unimodular lattices.} Let $\Lambda \in \mathcal{M}_1$ be a unimodular lattice of $\mathbf{R}^2$. A vector $v \in \Lambda$ is said to be primitive if it cannot be written as a non-negative integer multiple of another vector of $\Lambda$. Denote by $\Lambda_\mathrm{prim} \subseteq \Lambda$ the subset of all primitive vectors of $\Lambda$. Recall that $\| \cdot \|$ denotes the Euclidean norm of $\mathbf{R}^2$. For every $L > 0$ we consider the counting function
\begin{equation}
\label{eq:prim_lat_count}
p(\Lambda,L) := \{v \in \Lambda_\mathrm{prim}  \colon \| v \| \leq L \}.
\end{equation}

As one might already expect after considering the case of the integer lattice $\mathbf{Z}^2 \subseteq \mathbf{R}^2$, explicitely describing $p(\Lambda,L)$ for a given unimodular lattice $\Lambda \in \mathcal{M}_1$ as a function of $L$ is a particularly hard problem. Nevertheless, and perhaps surprisingly, the average
\[
\int_{\mathcal{M}_1} p(\Lambda,L) \thinspace d\widehat{\mu}(\Lambda)
\]
can be computed explicitly. Indeed, the following formula holds. This formula is a consequence of Siegel's famous integration formulas over $\mathcal{M}_1$ \cite{Si45}. In \S 4 we will prove this result using a general local change of variables formula. See Exercise \ref{ex:long}.

\begin{proposition}
	\label{prop:siegel_basic}
	The following integration formula holds,
	\[
	\int_{\mathcal{M}_1}  p(\Lambda,L) \thinspace d\widehat{\mu}(\Lambda) = 2 \cdot L^2.
	\]
\end{proposition}

To prove Theorem \ref{theo:main_lec_1} we will need a more uniform control on the integrability of the counting functions $p(\Lambda,L)$ as $L \to \infty$. To this end we consider the function $u \colon \mathcal{M}_1 \to \mathbf{R}$ given by
\[
u(\Lambda) := \sup_{v \in \Lambda} \frac{1}{\|v\|}.
\]

\begin{exercise}
	\label{eq:dct}
	Show there exists a constant $C > 0$ such that for every $\Lambda \in \mathcal{M}_1$ and every $L > 0$,
	\[
	p(\Lambda,L) \leq C \cdot L^2 \cdot u(\Lambda).
	\]
	Additionally, show that the function $u \colon \mathcal{M}_1 \to \mathbf{R}$ is integrable with respect to the measure $\widehat{\mu}$, i.e.,
	\[
	\int_{\mathcal{M}_1} u(\Lambda) \thinspace d\widehat{\mu}(\Lambda) < \infty.
	\]
\end{exercise}

\subsection*{Convergence and compactness.} Before proceeding with the proof of Theorem \ref{theo:main_lec_1} let us discuss a couple of basic convergence and compactness criteria. To prove the sequence of counting measure $(\nu^\mathrm{prim}_L)_{L > 0}$ on $\mathbf{R}^2$ converges in the weak-$\star$ topology we will use the following convergence criterion.

\begin{exercise}
	\label{ex:conv}
	Let $X$ be a metric space. Show that a sequence $(x_L)_{L>0}$ in $X$ converges to $x \in X$ if and only if every subsequence of $(x_L)_{L>0}$ has a subsubsequence converging to $x$.
\end{exercise}

The following compactness criterion for the weak-$\star$ topology serves as a convenient tool for extracting convergent subsequences of locally finite measures on metric spaces. This result is a consequence of the well known Banach-Alaoglu theorem in functional analysis.

\begin{theorem}
	\label{theo:ba_al}
	Let $X$ be a metric space and $(\mu_n)_{n \in \mathbf{N}}$ be a sequence of  locally finite measures on $X$. Suppose that the sequence $(\mu_n(K))_{n \in \mathbf{N}}$ is bounded for every $K \subseteq X$ compact. Then, the sequence $(\mu_n)_{n \in \mathbf{N}}$ has a subsequence converging in the weak-$\star$ topology to a locally finite measure on $X$.
\end{theorem}

\subsection*{Equidistribution of primitive integer points.} We are finally ready to prove that the sequence of counting measure $\smash{(\nu^\mathrm{prim}_L)_{L > 0}}$ on $\mathbf{R}^2$ converges in the weak-$\star$ topology to a constant multiple of the Lebesgue measure $\nu$. We will later deduce Theorem \ref{theo:main_lec_1} directly from this result.

\begin{theorem}
	\label{theo:equid_lat}
	With respect to the weak-$\star$ topology for measures on $\mathbf{R}^2$,
	\[
	\lim_{L \to \infty} \nu_{L}^\mathrm{prim} = \frac{6}{\pi^2} \cdot \nu.
	\]
\end{theorem}

\begin{proof}
	By Exercise \ref{ex:conv} applied to the space of regular measures on $\mathbf{R}^2$ with the weak-$\star$ topology, it is enough to show every subsequence of $\smash{(\nu^\mathrm{prim}_L)_{L > 0}}$ has a subsubsequence converging to $(6/\pi^2) \cdot \nu$ in the weak-$\star$ topology. Theorem \ref{theo:ba_al} and the bound in (\ref{eq:A3}) ensure that every subsequence of $\smash{(\nu^\mathrm{prim}_L)_{L > 0}}$ has a subsubsequence $\smash{(\nu^\mathrm{prim}_{L_k})_{k \in \mathbf{N}}}$ converging in the weak-$\star$ topology to a locally finite measure $\nu^\mathrm{prim}$ on $\mathbf{R}^2$. By Proposition \ref{prop:prim_meas_1}, it must be the case that $\nu^\mathrm{prim} = c \cdot \nu$ for some constant $c \geq 0$. It is enough then to show that $c = 6/\pi^2$, independent of the subsubsequence considered. To prove this identity holds we compute the following limit in two different ways
	\[
	\lim_{k \to \infty} \int_{\mathcal{M}_1} \frac{p(\Lambda,L_k)}{L_k^2} \thinspace d\widehat{\mu}(\Lambda).
	\]
	
	The first computation is immediate. Directly from Proposition \ref{prop:siegel_basic} we deduce
	\begin{equation}
	\label{eq:s1}
	\lim_{k \to \infty} \int_{\mathcal{M}_1} \frac{p(\Lambda,L_k)}{L_k^2} \thinspace d\widehat{\mu}(\Lambda) = 2.
	\end{equation}
	
	The second computation requires more work. For every matrix $A \in \mathrm{SL}(2,\mathbf{R})$ consider the subset
	\[
	B_A := \{v \in \mathbf{R}^2 \colon \|A \cdot v\| \leq 1\}.
	\]
	Notice that, for every unimodular lattice $\Lambda := A \cdot \mathbf{Z}^2 \in \mathcal{M}_1$ with $A \in \mathrm{SL}(2,\mathbf{Z})$ and every $k \in \mathbf{N}$,
	\begin{equation*}
	\frac{p(\Lambda,L_k)}{L_k^2} = \smash{\nu_{L_k}^\mathrm{prim}}(B_A).
	\end{equation*}
	In this setting, as $\lim_{k \to \infty} \smash{\nu^\mathrm{prim}_{L_k}} = c \cdot \nu$ in the weak-$\star$ topology, Theorem \ref{theo:port_1} ensures that
	\[
	\lim_{k \to \infty} \frac{p(\Lambda,L_k)}{L_k^2} = \lim_{k \to \infty} \smash{\nu^\mathrm{prim}_{L_k}}(B_A) = c \cdot \nu(B_A) = c \cdot \pi.
	\]
	From this identity, Exercise \ref{eq:dct}, the dominate convergence theorem, and the fact that $\widehat{\mu}(\mathcal{M}_1) = \pi/3$ we deduce
	\begin{equation}
	\label{eq:s2}
	\lim_{k \to \infty} \int_{\mathcal{M}_1} \frac{p(\Lambda,L_k)}{L_k^2} \thinspace d\widehat{\mu}(\Lambda) = c \cdot \frac{\pi^2}{3}.
	\end{equation}
	
	Putting together (\ref{eq:s1}) and (\ref{eq:s2}) we deduce
	\[
	c \cdot \frac{\pi^2}{3} = 2.
	\]
	Solving for $c$ we conclude $c = 6/\pi^2$, thus finishing the proof.
\end{proof}

\begin{exercise}
	\label{ex:lebesgue_lim}
	Recall the definition of the sequence of counting measures $(\nu_{L})_{L > 0}$ on $\mathbf{R}^2$ in (\ref{eq:leb_count}). Using the methods introduced in the proof of Theorem \ref{theo:equid_lat} show that (\ref{eq:lim_1}) holds, i.e., show that with respect to the weak-$\star$ topology for measures on $\mathbf{R}^2$,
	\[
	\lim_{L \to \infty} \nu_L = \nu.
	\]
\end{exercise}

\subsection*{Counting primitive integer points.} We are now ready to prove Theorem \ref{theo:main_lec_1}, the main result of this section.  Let us restate this theorem in the following equivalent way.

\begin{theorem}
	\label{theo:main_lec_1_red}
	The following asymptotic estimate holds,
	\[
	\lim_{L \to \infty} \frac{p\left(\mathbf{Z}^2,L\right)}{L^2} = \frac{6}{\pi}.
	\]
\end{theorem}

\begin{proof}
 Recall that if $B \subseteq \mathbf{R}^2$ is the Euclidean unit ball centered at the origin then, for every $L > 0$,
 \[
 \frac{p\left(\mathbf{Z}^2,L\right)}{L^2} = \smash{\nu_{L}^\mathrm{prim}}(B).
 \]
 Using Theorems \ref{theo:equid_lat} and \ref{theo:port_1} we conclude
  \[
 \lim_{L \to \infty} \frac{p\left(\mathbf{Z}^2,L\right)}{L^2} = \lim_{L \to \infty} \smash{\nu_{L}^\mathrm{prim}}(B) = \frac{6}{\pi^2} \cdot \nu(B) = \frac{6}{\pi}. \qedhere
 \]
\end{proof}

\begin{exercise}
	Using Theorem \ref{theo:equid_lat} show that for every unimodular lattice $\Lambda \in \mathcal{M}_1$,
	\[
	\lim_{L \to \infty} \frac{p(\Lambda,L)}{L^2} = \frac{6}{\pi}.
	\]
\end{exercise}

\section{Hyperbolic surfaces, Teichmüller spaces, and simple closed curves}

\subsection*{Outline of this section.} In this section we cover the background material needed to understand the proof of Theorem \ref{theo:main}. The focus will be in developing geometric intuition rather than on giving complete proofs. Unless otherwise stated, all surfaces considered will be connected and orientable. Two excellent references for the topics that will be covered in this section are \cite{FM11} and \cite{Mar16}.

\subsection*{The hyperbolic plane.} The hyperbolic plane $\mathbf{H}^2$ is the unique, up to isometry, two dimensional simply connected Riemannian manifold of constant sectional curvature $-1$. The hyperbolic plane can be modeled on the upper half space $\{z \in \mathbf{C} \ | \ \Im(z) > 0\}$ by endowing it with the Riemannian metric
\[
g := \frac{dx^2 + dy^2}{y^2}.
\]
The geodesics of this metric are the lines and half circles of the upper half space perpendicular to the the real axis $\mathbf{R} \subseteq \mathbf{C}$. See Figure \ref{fig:h2}. The orientation preserving isometries of this metric can be identified with the group $\mathrm{PSL}(2,\mathbf{R}) = \mathrm{SL}(2,\mathbf{R})/\{\pm \mathrm{I}\}$ acting on $\mathbf{H}^2$ by Möbius transformations:
\[
\left( \begin{array}{c c}
a & b \\ c & d
\end{array}\right).z := \frac{az +b}{cz+d}, \quad \forall z \in \mathbf{H}.
\]
This group acts simply transitively on the unit tangent bundle of $\mathbf{H}^2$. Such a large isometry group should hint at a rigid geometry: it becomes hard to distinguish objects up to isometry. The next exercise is a manifestation of this idea. Nevertheless, in dimension two the situation remains quite flexible, as we will see below. In higher dimensions the picture becomes incredibly rigid; curious readers are invited to investigate Mostow's rigidity theorem.

\begin{figure}[h!]
	\centering
	\includegraphics[width=.35\textwidth]{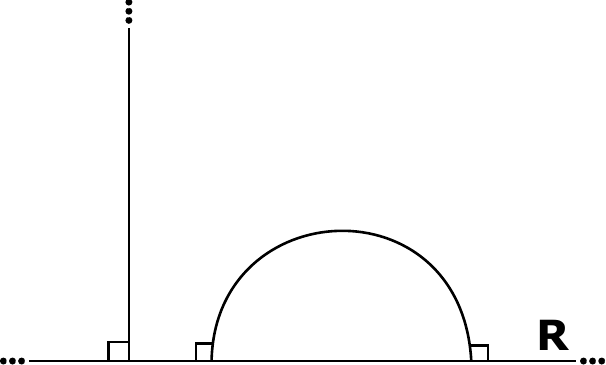}
	\vspace{+.2cm}
	\caption{The geodesics of the hyperbolic plane.} \label{fig:h2} 
\end{figure}

\begin{exercise}
	\label{ex:rigid_hex}
	Show that for every $(a,b,c) \in (\mathbf{R}^+)^3$ there exists a unique, up to isometry, hyperbolic right angled hexagon with alternating edge lengths $(a,b,c)$. \textit{Hint: Consider a configuration as in Figure \ref{fig:hexagon} and study how the length $z(y)$ varies as the parameter $y$ varies.}
\end{exercise}

\begin{figure}[h!]
	\centering
	\includegraphics[width=.40\textwidth]{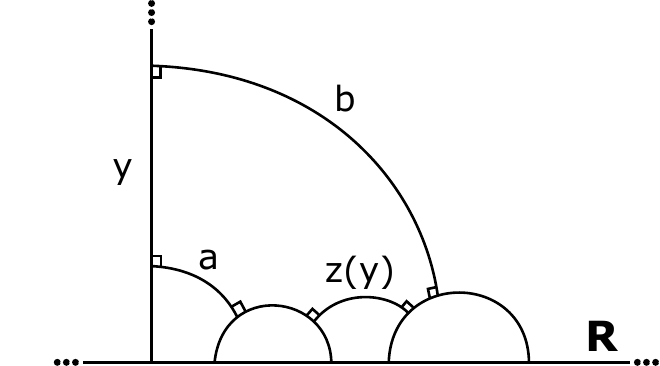}
	\vspace{+.2cm}
	\caption{Right angled hyperbolic hexagons are rigid.} \label{fig:hexagon} 
\end{figure}

\subsection*{Hyperbolic surfaces.} A hyperbolic surface is a surface whose geometry is locally modelled on $\mathbf{H}^2$. More concetely, a hyperbolic surface $X$ is a surface endowed with an atlas of charts to $\mathbf{H}^2$ whose transition functions are restrictions of isometries of $\mathbf{H}^2$. Pulling back the metric of $\mathbf{H}^2$ via these charts yields a metric of constant curvature equal to $-1$ on $X$. A geodesic of $X$ is a geodesic of this metric. Equivalently, a geodesic of $X$ is a curve which is mapped to geodesics of $\mathbf{H}^2$ via local charts. Every closed curve on $X$ can be tightened, i.e., homotoped, to a unique geodesic representative.

We will not spend much time discussing the many interesting features of hyperbolic surfaces but let us at least highlight one fact that is very useful to keep in mind for the sake of geometric intuition. The following fact is known as the collar lemma: every simple closed geodesic on a hyperbolic surface has a collar whose width goes to infinity as the length of the geodesic goes to zero. See Figure \ref{fig:collar_lemma}. 

\begin{figure}[h]
	\centering
	\begin{subfigure}[b]{0.4\textwidth}
		\centering
		\includegraphics[width=0.2\textwidth]{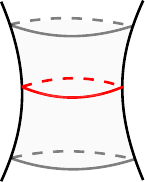}
			\vspace{+.8cm}
		\caption{Every geodesic has a collar.}
	\end{subfigure}
	\quad \quad \quad
	~ 
	\begin{subfigure}[b]{0.4\textwidth}
		\centering
		\includegraphics[width=0.2\textwidth]{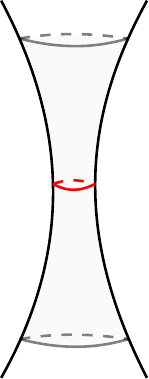}
		\caption{Shorter geodesics have longer collars.}
	\end{subfigure}
	\caption{The collar lemma.} 
	\label{fig:collar_lemma}
\end{figure}

\subsection*{Teichmüller space.} All hyperbolic surfaces of a given genus fit together nicely into a moduli space. To keep track of how any pair of hyperbolic surfaces looks with respect to each other we will record more information than just their position in moduli space. This leads us to introduce Teichmüller space. We will later see how to recover moduli space as a quotient of Teichmüller space.

For the rest of this lecture we fix an integer $g \geq 2$ and a connected, oriented, closed surface $S_g$ of genus $g$. The Teichmüller space $\mathcal{T}_{g}$ of marked hyperbolic structures on $S_g$ is the space of all pairs $(X,\varphi)$ where $X$ is an oriented hyperbolic surface and $\varphi \colon S_g \to X$ is an orientation preserving homeomorphism, modulo the equivalence relation $(X_1,\varphi_1) \sim (X_2,\varphi_2)$ if and only if there exists an orientation preserving isometry $I \colon X_1 \to X_2$ isotopic to $\varphi_2 \circ \varphi_1^{-1}$. In most situations we will omit the marking of a point $(X,\varphi) \in \mathcal{T}_g$ and denote it simply by $X \in \mathcal{T}_g$.

Roughly speaking, a point in Teichmüller space does not only keep track of a hyperbolic structure on $S_g$ but also of how $S_g$ is ``wearing" that hyperbolic structure. This allows us to communicate marked hyperbolic structures on $S_g$ in a very explicit way. For instance, given a marked hyperbolic structure $(X,\varphi) \in \mathcal{T}_g$, any parametrized closed curve $\gamma$ on $S_g$ can be canonically identified with the closed geodesic obtained by tightening the closed curve $\varphi(\gamma)$ on $X$ to its unique geodesics representative. We denote the length of this representative by $\ell_\gamma(X)$.

\subsection*{The mapping class group.} The mapping class group of $S_g$, denoted $\mcg$, is the group of orientation preserving homeomorphisms of $S_g$ up to homotopy. More explicitely, 
\[
\mcg := \mathrm{Homeo}^+(S_g) / \mathrm{Homeo}_0(S_g).
\] 
This group acts naturally on $\mathcal{T}_g$ by changing the markings: for every $(X,\varphi) \in \mathcal{T}_g$ and every $\phi \in \mcg$, 
\[
\phi.(X,\varphi) = (X,\varphi \circ \phi^{-1}).
\]
We think of $\mcg$ acting on $\tt$ as an analogue of $\mathrm{SL}(2,\mathbf{Z})$ acting on $\mathbf{H}^2$.

\begin{exercise}
	Let $(X,\varphi) \in \mathcal{T}_g$ be a marked hyperbolic structure on $S_g$. Show there exists a natural one-to-one correspondence between the $\mcg$ stabilizer of $(X,\varphi)$ and the set of isometries of $X$.
\end{exercise}

A particularly important family of mapping classes are Dehn twists. Given a simple closed curve $\gamma$ on $S_g$, the Dehn twist $T_\gamma$ of $S_g$ along $\gamma$ is the mapping class which leaves the surface $S_g$ untouched safe for an embedded annular neighborhood of $\gamma$ which gets twisted to the right, with respect to the orientation of $S_g$, by a full rotation. See Figure \ref{fig:dt}.

\begin{figure}[h!]
	\centering
	\vspace{+0.1cm}
	\begin{tabular}{c c c}
		\includegraphics[scale=.8]{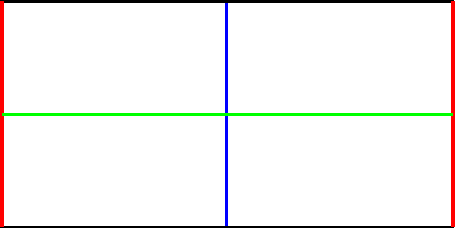} & \begin{tabular}{c} $\xrightarrow{ \quad T_\gamma \quad } $  \\ \\ \\ \\ \\ \end{tabular}& \begin{tabular}{c} \includegraphics[scale=.8]{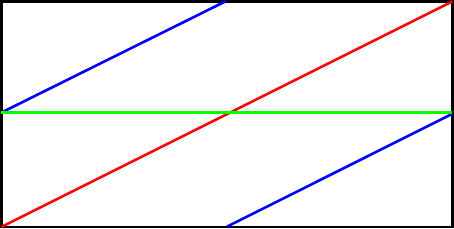} \\[1.5cm] \end{tabular}
	\end{tabular}
	\vspace{-1.55cm}
	\captionsetup{width=\linewidth}
	\caption{Dehn twist in an annular neighborhood of a simple closed curve $\gamma$ (in green).} \label{fig:dt}
\end{figure}

\subsection*{Moduli space of hyperbolic surfaces.} As $\mcg$ acts on $\mathcal{T}_g$ by changing the markings, one should expect that, if one quotients $\mathcal{T}_g$ by this action, one should get a space grouping together all (unmarked) hyperbolic surfaces of genus $g$. It turns out that $\mcg$ acts on $\mathcal{T}_g$ properly discontinuously and thus the corresponding quotient behaves nicely. The quotient $\mathcal{M}_g := \mathcal{T}_g / \mcg$ is the moduli space of genus $g$ hyperbolic surfaces. By uniformization, this space can be canonically identified with the moduli space of genus $g$ Riemann surfaces you might have seen in algebraic geometry.

\subsection*{Fenchel-Nielsen coordinates.} Notice that any orientable topological surface of genus $g \geq 2$ can be constructed by glueing $2g-2$ pairs of pants, that is, spheres with three boundary components, along their boundaries. See Figure \ref{fig:pants} for an example. A similar construction can also be considered for hyperbolic surfaces. Indeed, cutting an orientable hyperbolic surface $X$ of genus $g \geq 2$ along any collection of $3g-3$ disjoint simple closed geodesics yields $2g-2$ hyperbolic pairs of pants with geodesic boundary components. This class of pairs of pants is very rigid, as the following exercise shows.

\begin{figure}[h!]
	\centering
	\includegraphics[width=.4\textwidth]{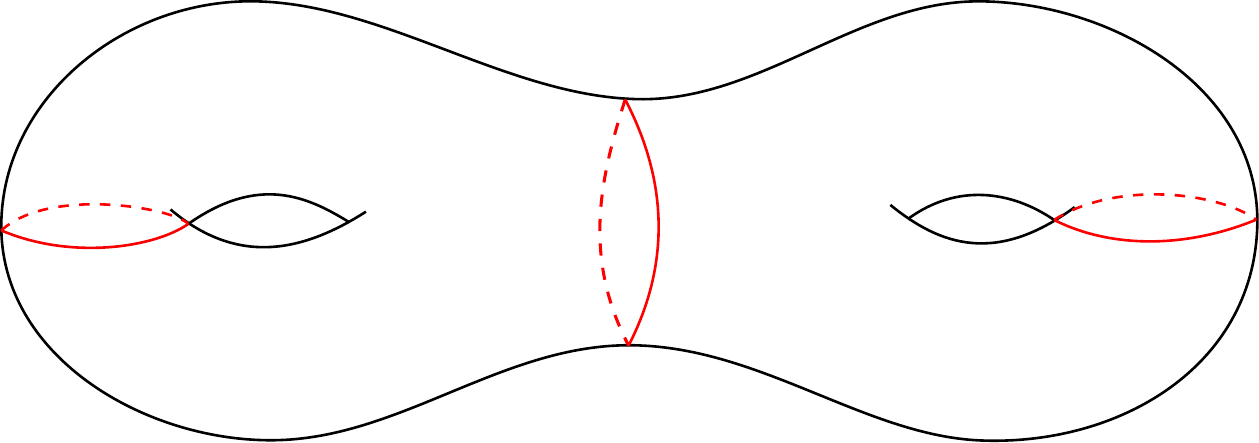}
	\caption{Pair of pants decomposition of a genus $2$ surface.} \label{fig:pants} 
\end{figure}

\begin{exercise}
	\label{ex:pants_rigid}
	Show that, for every $a,b,c \in \mathbf{R}^+$, there exists a unique, up to isometry, hyperbolic pair of paints with geodesic boundary components of lengths $a,b,c$. \textit{Hint: Cutting a hyperbolic pair of pants with geodesic boundary components along the orthogeodesics joining its boundary components yields a pair of isometric hyperbolic right-angled hexagons. See Figure \ref{fig:hex_cut}.}
\end{exercise}

\begin{figure}[h!]
	\centering
	\includegraphics[width=.25\textwidth]{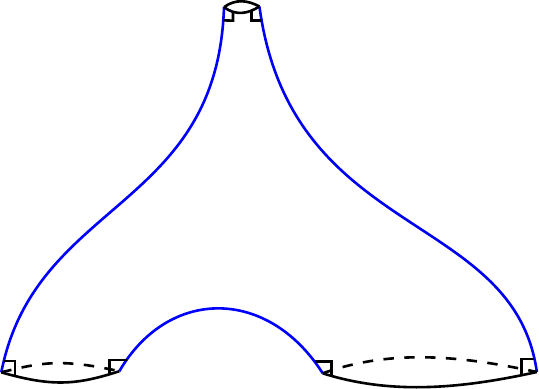}
	\caption{Cutting a hyperbolic pair of pants into isometric right-angled hexagons.} \label{fig:hex_cut} 
\end{figure}

By Exercise \ref{ex:pants_rigid}, the hyperbolic pairs of pants obtained by cutting an orientable hyperbolic surface $X$ of genus $g \geq 2$ along a maximal collection of disjoint simple closed geodesics are determined, up to isometry, by the lengths of the geodesics one cuts along. Glueing back these pairs of pants following the same original pattern allows us to recover $X$. One needs to be careful at this point as there are several ways in which one can glue back these pants along their cuffs. Indeed, for each geodesic one cutted $X$ along there is a full circle worth of different twist with which one can glue back the adjacent pairs of pants. If one is interested in marked hyperbolic surfaces rather than just hyperbolic surfaces, there is actually a full line worth of different twists one can glue with respect to.

More concretely, one can deform (marked) oriented hyperbolic surfaces using the following operation: given a (marked) oriented hyperbolic surface $X$, a simple closed geodesic $\gamma$ on $X$, and $t \in \mathbf{R}$, cut $X$ along $\gamma$ and glue the resulting surface back along $\gamma$ with a twist of $t$ units of hyperbolic length to the right with respect to the orientation of $X$. See Figure \ref{fig:FN_twist}. This operation is known as a Fenchel-Nielsen twist. Given a simple closed curve $\gamma$ on $S_g$ and a marked hyperbolic structure on $X \in \mathcal{T}_g$, the point $T_\gamma.X \in \mathcal{T}_g$ is equal to the marked hyperbolic structure obtained by doing a Fenchel-Nielsen twist of parameter $t = \ell_{\gamma}(X)$ along the unique geodesic representative of $\gamma$ on $X$.

\begin{figure}[h]
	\centering
	\begin{subfigure}[b]{0.4\textwidth}
		\centering
		\includegraphics[width=.85\textwidth]{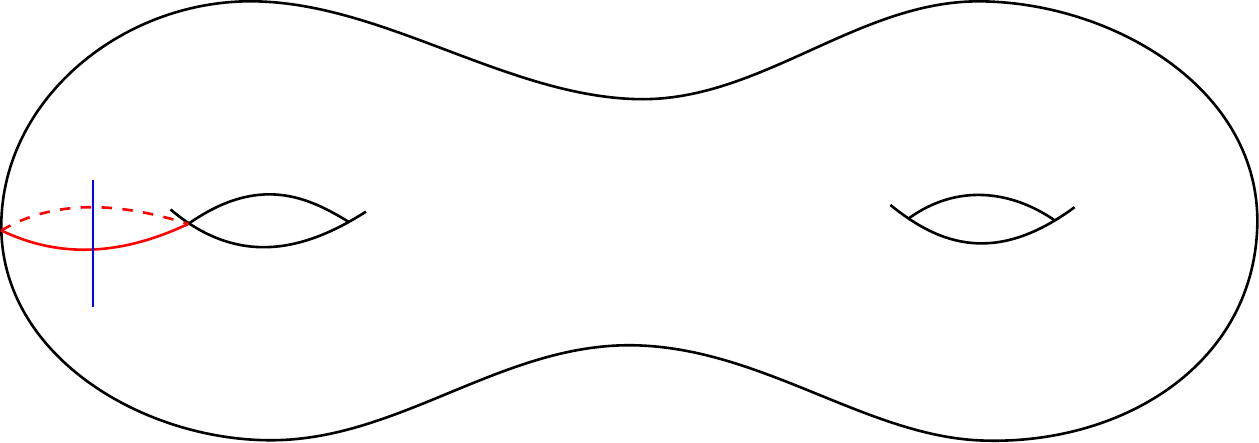}
		\caption{Before twist.}
	\end{subfigure}
	\quad \quad \quad
	~ 
	\begin{subfigure}[b]{0.4\textwidth}
		\centering
		\includegraphics[width=.85\textwidth]{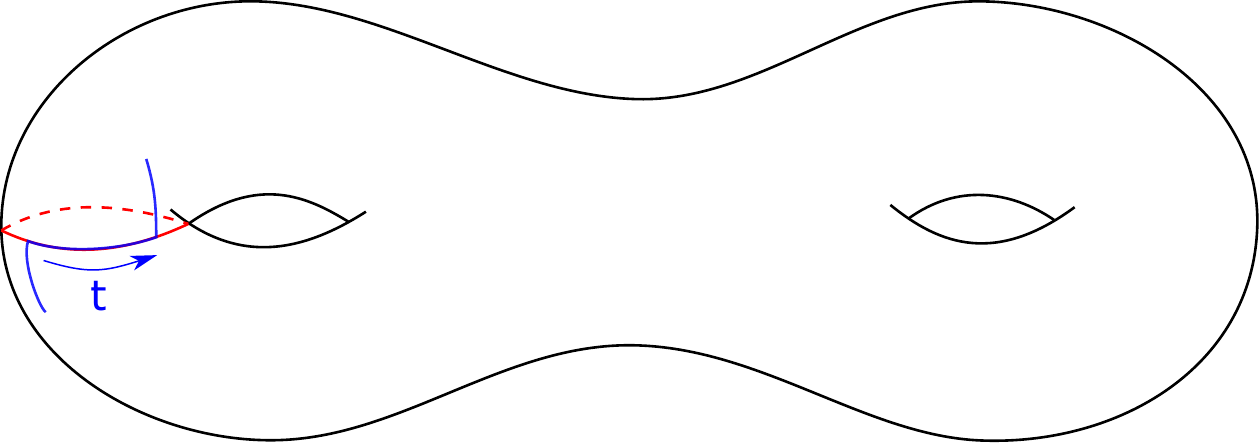}
		\caption{After twist.}
	\end{subfigure}
	\caption{Fenchel Nielsen twist along a simple closed curve (in red).} 
	\label{fig:FN_twist}
\end{figure}

The discussion above leads us to introduce Fenchel-Nielsen coordinates. A pair of pants decomposition of $S_g$ is a maximal collection of disjoint simple closed curves on $S_g$. Fix a pair of pants decomposition $\mathcal{P} := \smash{(\gamma_i)_{i=1}^{3g-3}}$ of $S_g$. Given a marked hyperbolic structure $X \in \mathcal{T}_g$, for every $i \in \{1,\dots,3g-3\}$ denote $\ell_i(X) := \ell_{\gamma_i}(X)$ and let $\tau_i(X)$ be the twist of $X$ at the geodesic representative of $\gamma_i$. The parameters $(\ell_i,\tau_i)_{i=1}^{3g-3} \in \smash{(\mathbf{R}^+ \times \mathbf{R})^{3g-3}}$ provide global coordinates for $\mathcal{T}_g$ known as Fenchel-Nielsen coordinates. In particular, $\mathcal{T}_g$ is homeomorphic to an open ball of dimension $6g-6$. For every $i \in \{1,\dots,3g-3\}$, the action of the Dehn twist $T_{\gamma_i}$ in Fenchel-Nielsen coordinates corresponds to leaving all coordinates constant safe for $\tau_i$ which gets changed by the transformation $\tau_i \mapsto \tau_i + \ell_i$.

\begin{exercise}
	Using the definition of Fenchel-Nielsen coordinates, the collar lemma, Dehn twists, and your geometric intuition, come up with an intuitive explanation of the following fact: a marked hyperbolic structure on Teichmüller space escapes to infinity, i.e., leaves every compact set, if and only if one of its geodesics becomes arbitrarily long. 
\end{exercise}

Let us highlight an interesting feature of geodesic pair of pants decompositions of hyperbolic surfaces. By work of Bers \cite{Ber85}, every closed, connected, oriented hyperbolic surface admits a geodesic pair of pants decompositions whose cuffs have lengths bounded by a linear function of $g$. Buser conjectured this result can be improved to ensure the cuffs have lengths bounded by a linear function of $\sqrt{g}$ \cite{Bus92}. This conjecture remains open until today, even for random hyperbolic surfaces.

\begin{exercise}
	\label{ex:pants}
	Show that there are finitely many pair of pants decompositions of $S_g$ up to homeomorphism. Can you give an asymptotic estimate for the number of such equivalence classes? \textit{Hint: If $I_N$ denotes the number of isomorphism classes of cubic multigraphs on $N$ vertices then, as $N \to \infty$,
	\[
	I_N \sim \frac{e^2}{\sqrt{\pi N}} \cdot \left( \frac{3N}{4e}\right)^{N/2}.
	\]}
\end{exercise}

\begin{exercise}
	\label{ex:compactness}
	Use Bers's theorem and Exercise \ref{ex:pants} to show that for every $\epsilon > 0$ the subset $K_\epsilon \subseteq \mathcal{M}_g$ of genus $g$ hyperbolic surfaces all of whose closed geodesics have length $\geq \epsilon$ is compact. This result is commonly known as Mumford's compactness criterion. \textit{Hint: Using Fenchel-Nielsen coordinates write $K_\epsilon \subseteq \mathcal{M}_g$ as a union of finitely many projections of compact subsets of $\mathcal{T}_g$.}
\end{exercise}

\subsection*{The Weil-Petersson measure on Teichmüller space.} The discussion above suggests considering the following volume form on Teichmüller space: Given a set of Fenchel-Nielsen coordinates $(\ell_i,\tau_i)_{i=1}^{3g-3} \in \smash{(\mathbf{R}^+ \times \mathbf{R})^{3g-3}}$, denote by $v_{\mathrm{wp}}$ the volume form on $\mathcal{T}_g$ given by
\begin{equation}
\label{eq:wol1}
v_{\mathrm{wp}} := \bigwedge_{i=1}^{3g-3} d\ell_i \wedge d\tau_i.
\end{equation}
This volume form is independent of the choice of Fenchel-Nielsen coordinates and in particular is mapping class group invariant. This fact is a consequence of a deep result of Wolpert known as Wolpert's magic formula \cite{Wol83}. We refer to the volume form $v_{\mathrm{wp}}$ above as the Weil-Petersson volume form of $\mathcal{T}_g$. We refer to the corresponding measure $\mu_{\mathrm{wp}}$ as the Weil-Petersson measure of $\mathcal{T}_g$.

\subsection*{Local pushforwards of measures.} In the upcoming discussion we make use of the following abstract measure theoretic construction. Let $X$ be a locally compact, Hausdorff, second countable topological space endowed with a properly discontinuous action of a discrete group $G$. Notice that $X/G$ is also locally compact, Hausdorff, and second countable. Let $\pi \colon X \to X/G$ be the corresponding quotient map. As the action of $G$ on $X$ is properly discontinuous, one can cover $X$ by open subsets $U \subseteq X$ invariant under the action of finite subgroups $\Gamma_U < G$ and such that $gU \cap U = \emptyset$ for all $g \in G \setminus \Gamma_U$. Open sets $U/\Gamma_U \subseteq X/G$ of this form will be refered to as well covered.

Given a locally finite $G$-invariant Borel measure $\mu$ on $X$, there exists a unique locally finite Borel measure $\pi_\#  \mu$ on $X/G$ satisfying the following property: If $U/\Gamma_U \subseteq X/G$ is a well covered open set, 
\begin{equation}
	\label{eq:stab_fact}
	\left(\pi_\#  \mu\right)|_{U/\Gamma_U} =  \smash{\frac{1}{|\Gamma_U|} \cdot \left(\pi|_U\right)_* \left(\mu|_U\right)},
\end{equation}
where $(\pi|_U)_* (\mu|_U)$ denotes the usual pushforward of the measure $\mu|_U$ through the map $\pi|_U$. We refer to the measure $\pi_\#  \mu$ as the local pushfoward of $\mu$ to $X/G$,

\begin{exercise}
	Check that the definition of $\pi_\#  \mu$  gives rise to a unique well defined measure on $X/G$.
\end{exercise}

\subsection*{The Weil-Petersson measure on moduli space.} Recall that the quotient $\mathcal{M}_g := \mathcal{T}_g /\mcg$ is the moduli space of genus $g$ hyperbolic surfaces. As the action of $\mcg$ on $\mathcal{T}_g$ is properly discontinuous and preserves the Weil-Petersson measure $\mu_{\mathrm{wp}}$, one can consider the local pushforward of $\mu_{\mathrm{wp}}$ to $\mathcal{M}_g$. We denote this measure by $\widehat{\mu}_{\mathrm{wp}}$ and refer to it as the Weil-Petersson measure of $\mathcal{M}_g$. 

\begin{exercise}
	\label{ex:wp_vol}
	Using Bers's theorem and Exercise \ref{ex:pants}, show that the Weil-Petersson measure $\widehat{\mu}_{\mathrm{wp}}$ on $\mathcal{M}_g$ is finite. Can you give a bound on the total Weil-Petersson measure $\widehat{\mu}_{\mathrm{wp}}(\mathcal{M}_g)$? \textit{Hint: Follow a similar approach as in Exercise \ref{ex:compactness}.}
\end{exercise}

Mirzakhani and Zograf \cite{MZ15} proved asymptotic estimates for the total Weil-Petersson volumes $\widehat{\mu}_{\mathrm{wp}}(\mathcal{M}_g)$ as $g \to \infty$ that wildly differ from the rough upper bound one obtains in Exercise \ref{ex:wp_vol}. The finiteness of the Weil-Petersson measure on moduli space suggests studying random hyperbolic surfaces sampled according to this measure. This has become a very active field of study in recent years. As a starting point for curious readers we recommend \cite{Mir13,MZ15,MP19,Mo20,MT20,LW21}.

\subsection*{Simple closed multi-curves.} We refer to a parametrized simple closed curve on $S_g$ up to homotopy and orientation reversal as a simple closed curve. We will always assume simple closed curves on $S_g$ are homotopically non-trivial. The group $\mcg$ acts naturally on these equivalence classes. Given a marked hyperbolic structure $X \in \mathcal{T}_g$, every simple closed curve on $S_g$ corresponds to a unique simple closed geodesic on $X$. Through this perspective, counting problems for simple closed geodesics on $X$ can be easily reformulated as counting problems for simple closed curves on $S_g$.

\begin{exercise}
	The $9g-9$ theorem, see for instance \cite[Theorem 10.7]{FM11}, guarantees that a marked hyperbolic structure $X \in \mathcal{T}_g$ is completely determined by its simple marked length spectrum, i.e., by the function which to every simple closed curve $\gamma$ on $S_g$ assigns the length $\ell_{\gamma}(X)$ of its unique geodesic representative with respect to $X$. Using this theorem and Dehn-Thurston coordinates show that the kernel of the action of $\mcg$ on $\mathcal{T}_g$ is equal to the kernel of the action of $\mcg$ on the set of simple closed curves on $S_g$.
\end{exercise}

More generally, a simple closed multi-curve on $S_g$ is a formal sum $\gamma = \smash{\sum_{i=1}^k} a_i \gamma_i$ of distinct simple closed curves on $S_g$ with real positive weights $a_i \in \mathbf{R}^+$.  Given a marked hyperbolic structure $X \in \mathcal{T}_g$ we define $\ell_{\gamma}(X) := \smash{\sum_{i=1}^k a_i \ell_{\gamma_i}(X)}$. An integral simple closed multi-curve on $S_g$ is a simple closed multi-curve all of whose weights are integral. The group $\mcg$ acts naturally on simple closed multi-curves by acting on each of the components of the formal sum. We think of $\mcg$ acting on integral simple closed multi-curves as an analogue of $\mathrm{SL}(2,\mathbf{Z})$ acting on $\mathbf{Z}^2$.

We say two simple closed multi-curves on $S_g$ have the same topological type if they belong to the same mapping class group orbit. More generally, we say two simple closed multi-curves on homeomorphic surfaces have the same topological type if there exists a homeomorphism between the surfaces mapping one simple closed multi-curve to the other. See Figure \ref{fig:type} for an example. We think of integral simple closed multi-curves on $S_g$ of a fixed topological type as an analogue of $\mathbf{Z}^2_{\mathrm{prim}}$.

\begin{figure}[h]
	\centering
	\begin{subfigure}[b]{0.4\textwidth}
		\centering
		\includegraphics[width=.85\textwidth]{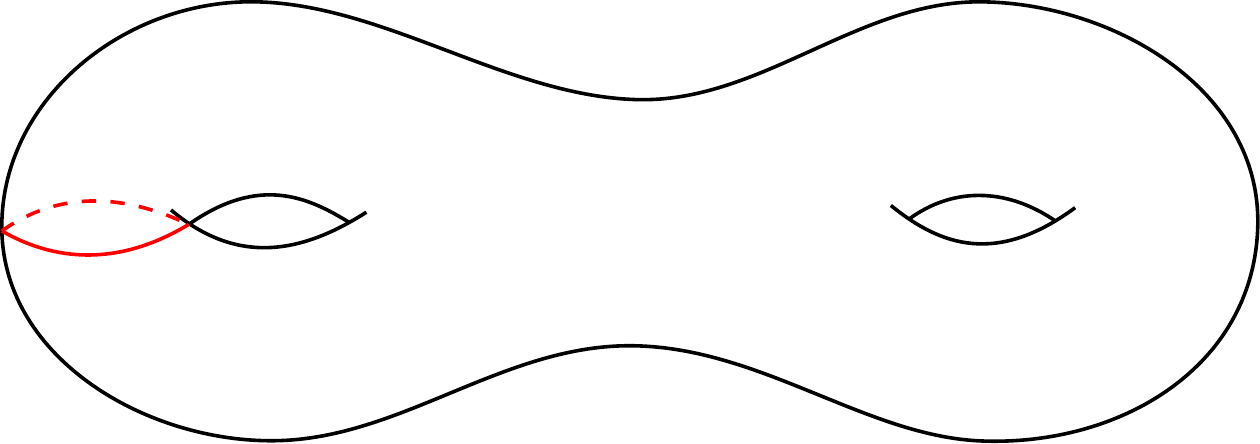}
		\caption{Non-separating simple closed curve.}
	\end{subfigure}
	\quad \quad \quad
	~ 
	\begin{subfigure}[b]{0.4\textwidth}
		\centering
		\includegraphics[width=.85\textwidth]{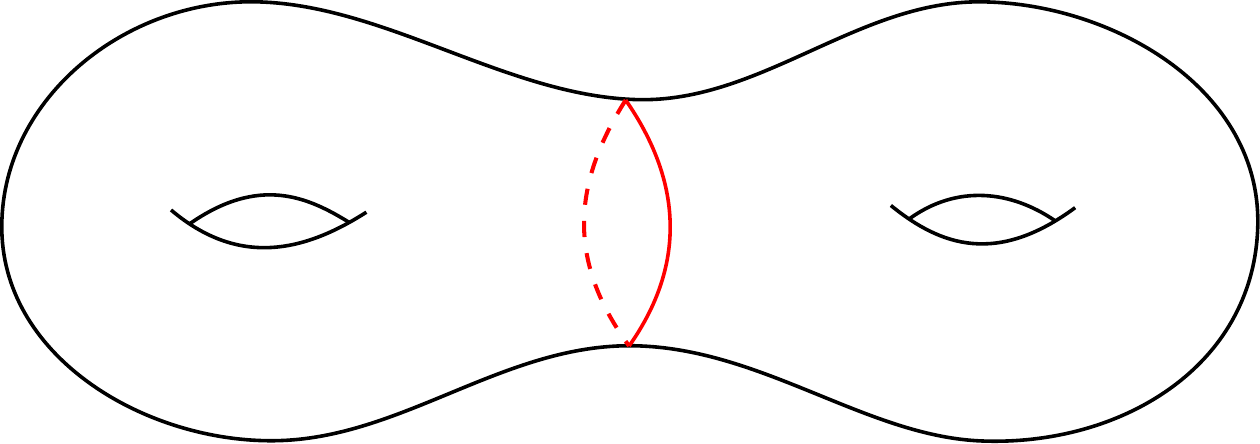}
		\caption{Separating simple closed curve.}
	\end{subfigure}
	\caption{Simple closed curves of different topological types on a genus $2$ surface.} 
	\label{fig:type}
\end{figure}

\subsection*{Measured geodesic laminations.} Just as the integer lattice $\mathbf{Z}^2$ sits inside the continuum $\mathbf{R}^2$, we would like to have a continuum interpolating between integral simple closed multi-curves on $S_g$. Although we will not need it explicitely, let us give a geometric description of how objects in this continuum look like. Fix a marked hyperbolic structure $X \in \mathcal{T}_g$. A geodesic lamination on $X$ is a closed subset of $X$ that can be written as a disjoint union of simple geodesics. The most basic example of a geodesic lamination is a union of disjoint simple closed geodesics. See Figure \ref{fig:lamination} for a more complicated example. Very commonly, the intersection of a geodesic lamination with a transverse arc is a fractal set, making it is hard to draw non-trivial examples.

\begin{figure}[h!]
	\centering
	\includegraphics[width=.4\textwidth]{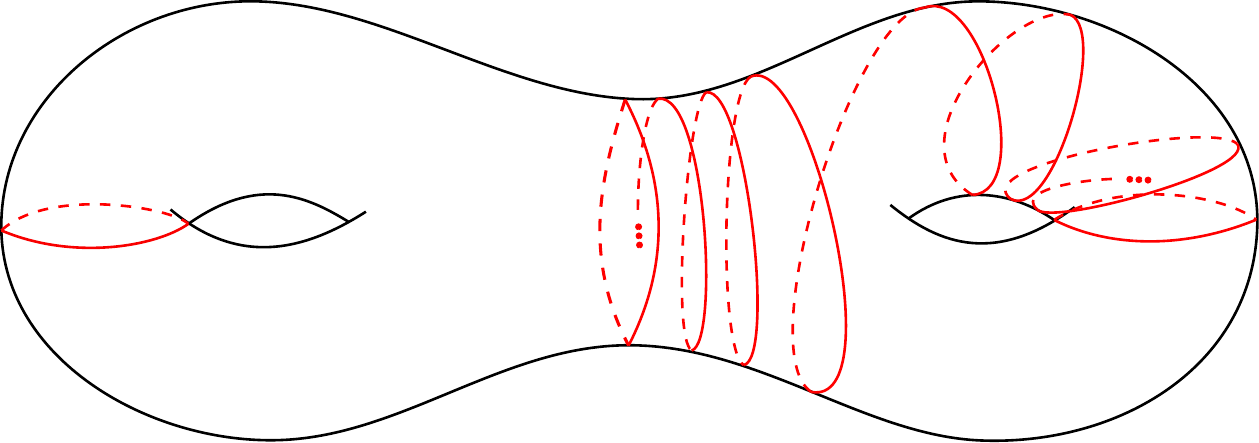}
	\caption{A non-trivial example of a geodesic lamination on a genus $2$ surface.} \label{fig:lamination} 
\end{figure}

A measured geodesic lamination is a geodesic lamination endowed with a fully supported invariant transverse measure. The transverse measure assigns a finite Borel measure to every arc transverse to the the lamination. This assignment is invariant under splitting of arcs and homotopies preserving the leaves of the lamination. Not every lamination admits a fully supported invariant transverse measure. 

\begin{exercise}
	Show that the geodesic lamination in Figure \ref{fig:lamination} does not admit a fully supported invariant transverse measure. \textit{Hint: Consider an arc going across the middle closed geodesic.}
\end{exercise} 

The different spaces of measured geodesic laminations obtained as $X$ varies over $\mathcal{T}_g$ can be canonically identified to each other. We will denote by $\mathcal{ML}_g$ any such space and refer to it as the space of measured geodesic laminations on $S_g$. Simple closed multi-curves embed naturally into $\ml$ by considering their geodesic representatives and endowing them with appropriately weighted transverse dirac measures. We denote by $\ml(\mathbf{Z}) \subseteq \ml$ the set of integral simple closed multi-curves on $S_g$. 

The space $\ml$ can be topologized in such a way that rationally weighted simple closed multi-curves on $S_g$ are dense in it. The action of $\mcg$ on simple closed multi-curves extends continuously to an action on $\ml$. We think of $\mcg$ acting on $\ml$ as an analogue of $\mathrm{SL}(2,\mathbf{Z})$ acting on $\mathbf{R}^2$. By work of Thurston \cite{T80}, the length $\ell_\lambda(X) > 0$ of a measured geodesic lamination $\lambda \in \ml$ with respect to a marked hyperbolic structure $X \in \mathcal{T}_g$ can be defined in a unique continuous way extending the definition on simple closed multi-curves introduced above.

\subsection*{Dehn-Thurston coordinates.} In a similar spirit to how hyperbolic surfaces can be constructed by glueing hyperbolic pairs of pants with geodesic boundary components, integral simple closed multi-curves on surfaces can be constructed by glueing simple arc systems on pairs of pants. By an arc on a pair of pants we will always mean an arc joining two of its boundary components. Every simple arc on a pair of pants is isotopic to exactly one of the six arcs represented in Figure \ref{fig:DT}. Every simple arc system on a pair of pants is isotopic to an arc system contructed by taking disjoint parallel copies of these arcs. Given an assignment of non-negative integers $(a,b,c) \in \mathbf{N}$ to the boundary components of a pair of pants satisfying $a+b+c \in 2\mathbf{N}$, there exists a unique, up to isotopy, arc system realizing $(a,b,c)$ as the number of intersections of the arc system with the boundary components.

\begin{figure}[h!]
	\centering
	\includegraphics[width=.25\textwidth]{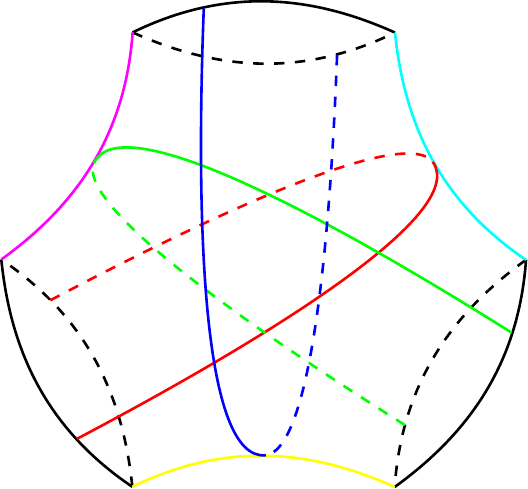}
	\caption{The six isotopy classes of simple arcs in a pair of pants.} \label{fig:DT} 
\end{figure}

Given simple arc systems on pairs of pants glued along two of their boundary components, these arc systems can be glued together if and only if the number of times they intersect the glued boundary components is the same. There exist $\mathbf{Z}$ many different ways of glueing these arc systems depending on how much we twist one with respect to the other before glueing them together.

Let $\mathcal{P}:=(\gamma_i)_{i=1}^{3g-3}$ be a pair of pants decomposition of $S_g$. Following the discussion above we can parametrize integral simple closed multi-curves on $S_g$ in terms of intersection numbers $m_i \in \mathbf{N}$ and twisting numbers $t_i \in \mathbf{Z}$ with respect to the components of $\mathcal{P}$. Consider the parameter space
\[
\Sigma:= \prod_{i=1}^{3g-3} \left(\mathbf{R}_{\geq 0} \times \mathbf{R}\right)/\sim,
\]
where $\sim$ denotes the equivalence relation on $\mathbf{R}_{\geq 0} \times \mathbf{R}$ identifying $(0,t) \sim (0,-t)$ for every $t \in \mathbf{R}$. The parameters $\smash{(m_i,t_i)_{i=1}^{3g-3}}$ give rise to a bijection between the set of integral simple closed multi-cuves on $S_g$ and the set of integral points in $\Sigma$ such that for every complementary region $R$ of $S_{g,n} \backslash \mathcal{P}$ the parameters $m_i$ corresponding to components $\gamma_i$ of $\mathcal{P}$ bounding $R$ add up to an even number. If $m_i = 0$ and $t_i = t \in \mathbf{Z}$, the corresponding simple closed multi-curve has $|t|$ disjoint parallel copies of $\gamma_i$. We refer to the parameters $\smash{(m_i,t_i)_{i=1}^{3g-3}}$ as a set of Dehn-Thurston coordinates.

Although we have not defined twisting numbers precisely, we can easily describe the effect of changing them by a given amount. See Figure \ref{fig:DT_twist} for an example. In particular, for every $i \in \{1,\dots,3g-3\}$, the action of the Dehn twist $T_{\gamma_i}$ in Dehn-Thurston coordinates corresponds to leaving all coordinates constant except for $t_i$ which gets changed by the transformation $t_i \mapsto t_i + m_i$. 

\begin{figure}[h]
	\centering
	\begin{subfigure}[b]{0.4\textwidth}
		\centering
		\includegraphics[width=.35\textwidth]{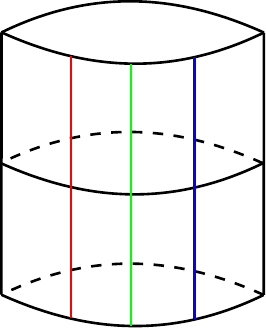}
		\caption{Before change.}
	\end{subfigure}
	\quad \quad \quad
	~ 
	\begin{subfigure}[b]{0.4\textwidth}
		\centering
		\includegraphics[width=.35\textwidth]{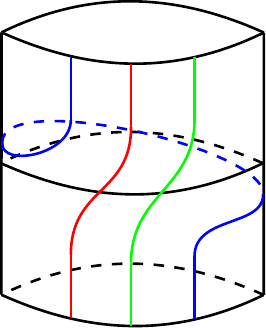}
		\caption{After change.}
	\end{subfigure}
	\caption{The effect of changing a twist coordinate by $t_i \mapsto t_i+1$.} 
	\label{fig:DT_twist}
\end{figure}

By work of Thurston \cite{T80}, the parametrization of $\ml(\mathbf{Z})$ provided by any set of Dehn-Thurston coordinates extends to a homeomorphism between $\ml$ and the parameter space $\Sigma$. This shows in particular that $\ml$ is homeomorphic $\mathbf{R}^{6g-6}$. For more details on the definition and properties of Dehn-Thurston coordinates we refer the reader to \cite[\S 1.2]{PH92}.

\subsection*{The Thurston measure.} Any set of Dehn-Thurston coordinates induces a homeomorphism between $\ml$ and the parameter space $\Sigma$. The Thurston measure $\mu_{\mathrm{Thu}}$ on $\ml$ is the pullback of the Lebesgue measure on $\Sigma$ via this identification. This measure is well defined, i.e., independent of the choice of Dehn-Thurston coordinates. Indeed, in analogy with the case of the Lebesgue measure, one can characterize $\mu_{\mathrm{Thu}}$ as the weak-$\star$ limit of a natural family of rescaled integral simple closed multi-curve counting measures. Consider the natural $\mathbf{R}^+$ action on $\ml$ scaling transverse measures. This action corresponds to the natural scaling action of $\mathbf{R}^+$ on $\Sigma$ via the identification by Dehn-Thurston coordinates. For every $L > 0$ consider the rescaled counting measure
\[
\mu_L := \frac{1}{L^{6g-6}} \sum_{\gamma \in \ml(\mathbf{Z})} \delta_{\frac{1}{L} \cdot \gamma}.
\]

\begin{exercise}
	\label{ex:thu_1}
	Show that the sequence of rescaled counting measures $(\mu_L)_{L > 0}$ on $\ml$ converges in the weak-$\star$ topology as $L \to \infty$ to the pullback of the Lebesgue measure on $\Sigma$ under the identification with $\ml$ induced by any set Dehn-Thurston coordinates.
\end{exercise}

Directly from Exercise \ref{ex:thu_1} we deduce $\mu_{\mathrm{Thu}}$ is invariant with respect to the $\mcg$ action on $\ml$. The definition of $\mu_{\mathrm{Thu}}$ in terms of Dehn-Thurston coordinates guarantees it satisifies the following scaling property: for every Borel measurable subset $A \subseteq \mathcal{ML}_{g}$ and every $t > 0$,
\begin{equation}
\label{eq:thu_meas_scale}
\mu_{\text{Thu}}(t \cdot A) = t^{6g-6} \cdot \mu_{\text{Thu}}(A).
\end{equation}

\begin{exercise}
	\label{eq:boundary}
	Let $f \colon \ml \to \mathbf{R}$ be a non-negative, continuous function. Suppose that $f$ is homogeneous with respect to the $\mathbf{R}^+$ scaling action on $\ml$. Using (\ref{eq:thu_meas_scale}) show that, for every $c > 0$, 
	\[
	\mu(f^{-1}(\{c\})) = 0.
	\]
\end{exercise}

\subsection*{Analogies with the lattice setting.} Let us finish this section by highlighting some analogies that will be useful to have in mind when proving Theorem \ref{theo:main}. The proof of Theorem \ref{theo:main} will follow the same approach as the proof of Theorem \ref{theo:main_lec_1}, but will use the analogies in Table \ref{tb:1} together with important results of Mirzakhani that will be discused in \S 4.

\begin{table}[h]
	\begin{tabular}{c | c }
		Lattice setting & Surface setting \\
		\hline \hline 
		$\text{SL}(2,\mathbf{Z})$ & $\text{Mod}_g$ = mapping class group of $S_g$ \\
		$\mathbf{R}^2$  & $\mathcal{ML}_g$ = measured geodesic laminations on $S_g$\\
		$\mathbf{Z}^2_{\text{prim}}$ & $\text{Mod}_g \cdot \gamma$ = mapping class group orbit\\
		$\text{Leb}_{\mathbf{R}^2}$ & $\mu_{\text{Thu}}$ = Thurston measure \\
		$B_A$ & $B_X = \{\lambda \in \mathcal{ML}_g \ | \ \ell_\lambda(X) \leq 1\}$\\
	\end{tabular}
	\caption{Analogies between the lattice and surface settings.}
	\label{tb:1}
\end{table}

\section{Weil-Petersson volumes and Mirzakhani's integration formulas}

\subsection*{Outline of this section.} In this section we discuss two remarkable results of Mirzhakhani that will play a crucial role in the proof of Theorem \ref{theo:main}. The first of these results corresponds to the fact that the total Weil-Petersson volume of any moduli space of hyperbolic surfaces with totally geodesic boundary components is a polynomial on the lengths of the boundary components. See Theorem \ref{theo:vol_pol}. The second of these results corresponds to Mirzakahani's famous integration formula over moduli space. See Exercise \ref{ex:mir_int}. We focus most of our attention on the second of these results. We sketch a proof of the first of these results for tori with one boundary component at the end of this section.

\subsection*{Surfaces with boundary.} The first result we discuss concerns moduli spaces of hyperbolic surfaces with totally geodesic boundary components and their total Weil-Petersson volumes. To state this result precisely we first introduce some notation. As many of the definitions that follow are analogous to the ones discussed in the previous section, we do not spend much time covering them.

Let $g,b \geq 0$ be a pair of non-negative integers such that $2 - 2g - b < 0$. Fix a connected, oriented, compact surface $\smash{S_{g,b}}$ of genus $g$ with $b$ labeled boundary components $\beta_1,\dots,\beta_{b}$.  Denote by $\smash{\text{Mod}_{g,b}}$ the mapping class group of $\smash{S_{g,b}}$, that is, the group of orientation preserving homeomorphisms of $\smash{S_{g,b}}$ that setwise fix each boundary component, up to homotopy setwise fixing each boundary component.

Let $\smash{\mathbf{L}:= (L_i)_{i=1}^{b}} \in \smash{(\mathbf{R}_{>0})^{b}}$ be a vector of positive real numbers. Denote by $\mathcal{T}_{g,b}(\mathbf{L})$ the Teichmüller space of  marked, oriented hyperbolic structures on $S_{g,b}$ with labeled geodesic boundary components whose lengths are given by $\mathbf{L}$, up to isotopy fixing each boundary component setwise. The mapping class group $\smash{\text{Mod}_{g,b}}$ acts properly discontinuously on $\mathcal{T}_{g,b}(\mathbf{L})$ by changing the markings. The quotient $\mathcal{M}_{g,b}(\mathbf{L}) := \mathcal{T}_{g,b}(\mathbf{L}) / \text{Mod}_{g,b}$ is the moduli space of oriented hyperbolic structures on $S_{g,b}$ with labeled geodesic boundary components whose lengths are given by $\mathbf{L}$. 

\subsection*{Weil-Petersson volumes.} The Teichmüller space $\mathcal{T}_{g,b}(\mathbf{L})$ can be parametrized using Fenchel-Nielsen coordinates in a similar way as for closed surfaces. More concretely, for any pair of pants decomposition $\mathcal{P}:= \smash{(\gamma_i)_{i=1}^{3g-3+b}}$ of $S_{g,b}$, the length and twist parameters $(\ell_i,\tau_i)_{i=1}^{3g-3+b}$ of marked hyperbolic structures $X \in \mathcal{T}_{g,b}(\mathbf{L})$ with respect to the components of $\mathcal{P}$ provide a global coordinate system for $\mathcal{T}_{g,b}(\mathbf{L})$. The Weil-Petersson volume form $v_{\mathrm{wp}}$ of $\mathcal{T}_{g,b}(\mathbf{L})$ can be defined as follows,
\begin{equation}
\label{eq:wol2}
v_{\mathrm{wp}} := \bigwedge_{i=1}^{3g-3+b} d\ell_i \wedge d\tau_i.
\end{equation}
By work of Wolpert \cite{Wol83}, this volume form is well defined, independent of the choice of Fenchel-Nielsen coordinates, and in particular mapping class group invariant. We refer to the measure induced by $v_\mathrm{wp}$ as the Weil-Petersson measure of $\mathcal{T}_{g,b}(\mathbf{L})$ and denote it by $\mu_{\mathrm{wp}}$.

On the moduli space $\mathcal{M}_{g,b}(\mathbf{L}) := \mathcal{T}_{g,b}(\mathbf{L}) / \text{Mod}_{g,b}$ consider the local pushforward $\widehat{\mu}_{\mathrm{wp}}$ of the Weil-Petersson measure $\mu_{\mathrm{wp}}$. We refer to this measure as the Weil-Petersson measure of $\mathcal{M}_{g,b}(\mathbf{L})$. The same arguments suggested in Exercise \ref{ex:wp_vol} can be used to show that the Weil-Petersson measure $\widehat{\mu}_{\mathrm{wp}}$ on $\mathcal{M}_{g,b}(\mathbf{L})$ is finite. The total Weil-Petersson volume of $\mathcal{M}_{g,b}(\mathbf{L})$ will be denoted by
\[
V_{g,b}(\mathbf{L}) := \widehat{\mu}_{\mathrm{wp}}(\mathcal{M}_{g,b}(\mathbf{L})).
\]
By Exercise \ref{ex:pants_rigid}, if the surface $S_{g,b}$ is a pair of pants, i.e., if $g = 0$ and $b = 3$, the moduli space $\mathcal{M}_{g,n}^{b}(\mathbf{L})$ consists of exactly one point. In this case we adopt the convention 
\[
V_{0,3}(\mathbf{L}) := 1.
\]

The following remarkable theorem of Mirzakhani shows that the total Weil-Petersson volumes $V_{g,b}(\mathbf{L}) $ are polynomials on the $\mathbf{L}$ variables. We will later sketch a proof of this result for the case of tori with one boundary component, i.e., for the case $g=1$ and $b=1$.

\begin{theorem} \cite[Theorem 1.1]{Mir07a} \cite[Theorem 1.1]{Mir07c}
	\label{theo:vol_pol}
	Let $g, b \geq 0$ be non-negative integers such that $2 - 2g - b < 0$. The total Weil-Petersson volume
	\[
	V_{g,b}(L_1,\dots,L_{b})
	\]
	is a polynomial of degree $3g-3+b$ on the variables $L_1^2,\dots,L_{b}^2$. Moreover, if 
	\[
	V_{g,b}(L_1,\dots,L_{b}) = \sum_{\substack{\alpha \in (\mathbf{Z}_{\geq 0})^{b}\\ |\alpha| \leq 3g-3+b} } c_\alpha \cdot L_1^{2\alpha_1} \cdots L_{b}^{2\alpha_{b}},
	\]
	where $|\alpha| := \alpha_1 + \cdots + \alpha_{b}$ for every $\alpha \in (\mathbf{Z}_{\geq 0})^{b}$, then $c_\alpha \in \mathbf{Q}_{>0} \cdot \pi^{6g-6 +2b - 2|\alpha|}$. In particular, the leading coefficients of $V_{g,b}(L_1,\dots,L_{b})$ belong to $\mathbf{Q}_{> 0}$.
\end{theorem}

\subsection*{The local change of variables formula.} We now introduce a local change of variables formula that will be used in forthcoming discussions. Let $X$ be a locally compact, Hausdorff, second countable topological space endowed with a properly discontinuous action of a discrete group $G$ and $f \colon X \to \mathbf{R}$ be a measurable, non-negative function. Suppose that $f$ is invariant with respect to a subgroup $H < G$, that is, $f(h.x) = f(x)$ for every $x \in X$ and every $h \in H$. Denote by $\smash{\widetilde{f}} \colon X/H \to \mathbf{R}$ the measurable, non-negative function induced by $f$ on $X/H$. Consider the measurable non-negative function $\smash{\widehat{f}} \colon X/G \to \mathbf{R}$ which to every $x \in X/G$ assigns the value
\[
\widehat{f}(x) := \sum_{g \in G/H} f(g.x).
\]
Let $\mu$ be a locally finite $G$-invariant Borel measure on $X$. Denote by $\smash{\widetilde{\mu}}$ the local pushforward of $\mu$ to $X/H$ and by $\smash{\widehat{\mu}}$ the local pushforward of $\mu$ to $X/G$. The following exercise can be interpreted as a local change of variables formula.

\begin{exercise}
	\label{ex:local_cv}
	Show that, in the setting described above, the following integration formula holds,
	\[
	\int_{X/G} \widehat{f}(x) \thinspace d\widehat{\mu}(x) = \int_{X/H} \widetilde{f}(y) \thinspace d\widetilde{\mu}(y).
	\]
	\textit{Hint: Use a partition of unity of $X/G$ to reduce to the case where f is supported on the preimage of a well covered neighborhood.}
\end{exercise}

The following exercise, which corresponds to Proposition \ref{prop:siegel_basic}, is a classical application of the local change of variables formula in the setting of lattices.

\begin{exercise}
	\label{ex:long}
	Let $X := \mathrm{SO}(2,\mathbf{R})\backslash\mathrm{SL}(2,\mathbf{R})$ and $G := \mathrm{SL}(2,\mathbf{Z})$ acting on $X$ via right multiplication. Recall that $X$ can be identified with the space of marked, oriented, unimodular lattices on $\mathbf{R}^2$ up to rotation via the map which sends $A \in X$ to $A \cdot \mathbf{Z}^2 \subseteq \mathbf{R}^2$. Denote by $\| \cdot \|$ the Euclidean norm on $\mathbf{R}^2$ and by $e_1 := (1,0) \in \mathbf{R}^2$ the first canonical vector of $\mathbf{R}^2$. For every $L > 0$ consider the function $f \colon X \to \mathbf{R}$ which to every $A \in X$ assigns the value $f(A) := \mathbbm{1}_{\| A \cdot e_1\| \leq L}$. Notice that $f$ is invariant with respect to the subgroup $H < G$ generated by the unipotent matrix
	\[
	u_1 := \left(\begin{array}{c c}
	1 & 1 \\
	0 & 1
	\end{array} \right).
	\]
	Indeed, $H$ is the stabilizer of $e_1$ with respect to the linear action of $G$ on $\mathbf{R}^2$. Recall that $\mathcal{M}_1 := X/G$ can be identified with the space of oriented, unimodular lattices on $\mathbf{R}^2$ up to rotation. Recall the definition of the counting function $p(\Lambda,L)$ in (\ref{eq:prim_lat_count}). Show that for every $\Lambda \in \mathcal{M}_1$ and every $L > 0$,
	\[
	\smash{\widehat{f}}(\Lambda) = p(\Lambda, L).
	\]
	Recall that $X$ can be identified with the upper half space $\mathbf{H}^2$ via the map which sends $z \in \mathbf{H}^2$ to the unimodular lattice $\Lambda_z := \mathrm{span}_\mathbf{Z}(c_z e_1, c_z z)\subseteq \mathbf{R}^2$, where $c_z := \Im(z)^{-1/2} > 0$. Show that if $A \in X$ gets identified with $z \in \mathbf{H}^2$ then the following identity holds,
	\[
	\|A \cdot e_1\| = \Im(z)^{-1/2}.
	\]
	In particular, this implies that for every $L > 0$,
	\[
	f(A) = \mathbf{1}_{\Im(z) \geq 1/L^2}.
	\]
	Recall that the identification of $X$ with $\mathbf{H}^2$ conjugates the right action of $G$ on $X$ with the left action of $G$ on $\mathbf{H}^2$ by Möbius transformations. This identification sends the Haar measure $\mu$ on $X$ to the measure $y^{-2} dx dy$ on $\mathbf{H}^2$. Denote by $\widehat{\mu}$ the local pushforward of $\mu$ to $\mathcal{M}_1$ as defined in (\ref{eq:stab_fact}). Notice that this definition differs from the one used in Proposition \ref{prop:siegel_basic} by a factor of $2$. Using Exercise \ref{ex:local_cv}, the fundamental domain in Figure \ref{fig:mod}, and the discussion above, show that 
	\[
	\int_{\mathcal{M}_1} p(\Lambda,L)  \thinspace d\widehat{\mu}(\Lambda)= L^2.
	\]
\end{exercise}

\subsection*{The cut and glue fibration.} We now discuss a remarkable observation due to Mirzakhani regarding how the Weil-Petersson measure on certain quotients of Teichmüller space disintegrates along the fibers of natural projections over products of moduli spaces of lower complexity. This observation is nothing more than a simple consequence of Wolpert's magic formula but its importance cannot be understated.

For the rest of this discussion we fix an integer $g \geq 2$ and a connected, oriented, closed surface $S_g$ of genus $g$. Recall that $\mcg$ denotes the mapping class group of $S_g$. Given a simple closed curve $\alpha$ on $S_{g}$ denote by
$
\text{Stab}_0(\alpha) \subseteq \text{Mod}_{g}
$
the subgroup of mapping classes of $S_{g}$ that fix $\alpha$ up to isotopy together with its orientations. Although $\alpha$ is unoriented, it admits two possible orientations. A mapping class belongs to $\text{Stab}_0(\alpha)$ if it sends each orientation of $\alpha$ back to itself. An ordered simple closed multi-curve on $S_g$ is a tuple $\smash{\gamma := (\gamma_i)_{i=1}^{k}}$ of pairwise disjoint, pairwise non-isotopic simple closed curves on $S_g$. Given such a simple closed multi-curve on $S_{g}$ denote
\[
\text{Stab}_0(\gamma) := \bigcap_{i=1}^k \text{Stab}_0(\gamma_i) \subseteq \text{Mod}_{g}.
\]

For the rest of this discussion we fix an ordered simple closed multi-curve $\gamma := (\gamma_i)_{i=1}^k$ on $S_{g}$. Recall that $\mathcal{T}_g$ denotes the Teichmüller space of marked hyperbolic structures on $S_g$. The quotient 
$
\mathcal{T}_{g}/\text{Stab}_0(\gamma)
$
fibers naturally over a product of moduli spaces of surfaces with boundary of less complexity than $S_{g}$. This fibration, which we refer to as the cut and glue fibration of $\mathcal{T}_{g}/\text{Stab}_0(\gamma)$, will play a crucial role in our discussion of Mirzakhani's integration formula. To describe this fibration in detail we first introduce some notation.

Let $S_{g}(\gamma)$ be the potentially disconnected oriented surface with boundary obtained by cutting $S_{g}$ along the components of $\gamma$. See Figure \ref{fig:cut} for an example. Let $c > 0$ be the number of components of $S_g(\gamma)$. Fixing an orientation on each component of $\gamma$, we can keep track of which components of $S_{g}(\gamma)$ lie to the left and to the right of each component of $\gamma$, so we can label the components of $S_{g}(\gamma)$ in a consistent way, say $\Sigma_j$ with $j \in \{1,\dots,c\}$. As the components of $\gamma$ are labeled and oriented, this induces a labeling of the boundary components of each $\Sigma_j$. Let $g_j,b_j \geq 0$ with $2 - 2g_j - b_j < 0$ be a pair of non-negative integers such that $\Sigma_j$ is homeomorphic to $S_{g_j,b_j}$. Fix a homeomorphism between these surfaces respecting the labeling of their boundary components. 

\begin{figure}[h]
	\centering
	\begin{subfigure}[b]{0.4\textwidth}
		\centering
		\includegraphics[width=.85\textwidth]{pictures/tt_1.pdf}
		\caption{Before cutting.}
	\end{subfigure}
	\quad \quad \quad
	~ 
	\begin{subfigure}[b]{0.4\textwidth}
		\centering
		\includegraphics[width=.85\textwidth]{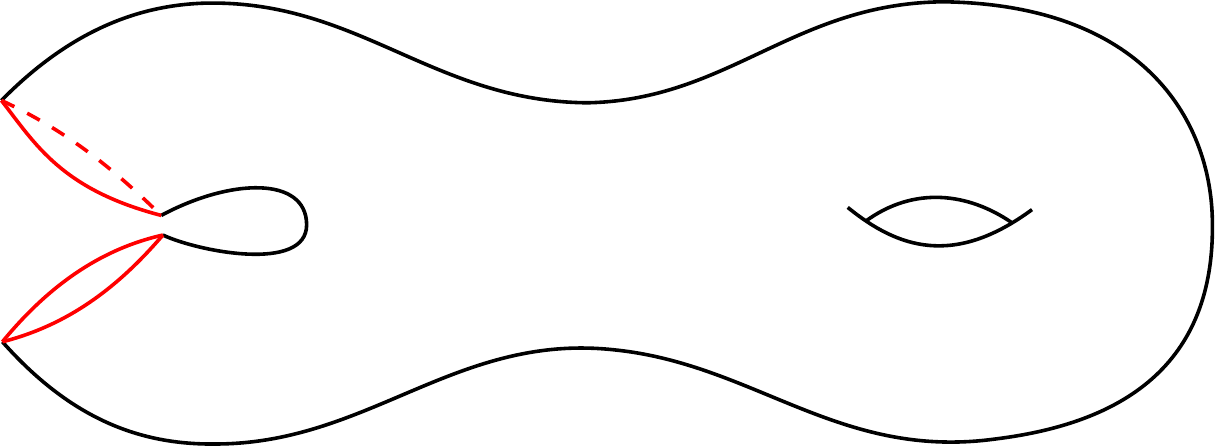}
		\caption{After cutting.}
	\end{subfigure}
	\caption{Cutting a genus $2$ surface along a non-separating simple closed curve.} 
	\label{fig:cut}
\end{figure}

The base of the cut and glue fibration is the space $\Omega_{g}(\gamma)$ of pairs $(\mathbf{L},\mathbf{X})$ where
\begin{gather*}
\mathbf{L} := (\ell_i)_{i=1}^k \in (\mathbf{R}_{>0})^k,\\
\mathbf{X} := (X_j)_{j=1}^c \in \prod_{j=1}^c \mathcal{M}_{g_j,b_j}(\mathbf{L}_j),
\end{gather*}
such that $\mathbf{L}_j \in (\mathbf{R}_{>0})^{b_j}$ is defined using the vector $\mathbf{L} := (\ell_i)_{i=1}^k \in (\mathbf{R}_{>0})^k$ and the correspondence between the labeling of the boundary components of $\Sigma_j$ and the labeling of the components of $\gamma$. 

An ordered simple closed multi-geodesic $\alpha := (\alpha_i)_{i=1}^k$ on a hyperbolic surface $X$ is a tuple of pairwise disjoint simple closed geodesics on $X$. Every ordered simple closed multi-geodesic is in particular an ordered simple closed multi-curve. Two ordered simple closed multi-curves on homeomorphic surfaces are said to have the same topological type if there exists a homeomorphism between the surfaces mapping one simple closed multi-curve to the other respecting their labelings. Recall that $\mathcal{M}_g$ denotes the moduli space of genus $g$ hyperbolic surfaces. The following exercise will play a crucial role in our description of the cut and glue fibration.

\begin{exercise}
	\label{ex:quot}
	Show there exists a one-to-one correspondence between points in the quotient space $\mathcal{T}_{g} /\text{Stab}_0(\gamma)$ and pairs $(X,\alpha)$ where $X \in \mathcal{M}_{g}$ is a genus $g$ hyperbolic surface and $\alpha := (\alpha_i)_{i=1}^k$ is a simple closed multi-geodesic on $X$ of the same topological type as $\gamma$ with a choice of orientation on each of its components, modulo the equivalence relation $(X',\alpha') \sim (X'',\alpha'')$ if and only if there exists an orientation-preserving isometry $I \colon X' \to X''$ sending the components of $\alpha'$ to the components of $\alpha''$ respecting their labelings and orientations. 
\end{exercise}

Consider the identification of the quotient space $\mathcal{T}_{g,n} /\text{Stab}_0(\gamma)$ provided by Exercise \ref{ex:quot}. The cut and glue fibration of $\mathcal{T}_{g,n} /\text{Stab}_0(\gamma)$ is the map $\Psi \colon \mathcal{T}_{g,n} /\text{Stab}_0(\gamma) \to \Omega_{g,n}(\gamma)$ which to every equivalence class $(X,\alpha) \in \mathcal{T}_{g,n} /\text{Stab}_0(\gamma)$ assigns the pair $(\mathbf{L},\mathbf{X}) \in \Omega_g(\gamma)$ given by
\begin{align*}
\mathbf{L} := \smash{(\ell_{\alpha_i}(X))_{i=1}^k}, \quad
\mathbf{X} := \smash{(X(\alpha)_j)_{j=1}^c},
\end{align*}
where $X(\alpha)_j \in \mathcal{M}_{g_j,b_j}(\mathbf{L}_j)$ is the $j$-th component (according to the labeling and orientation of the components of $\alpha$) of the hyperbolic surface with geodesic boundary components obtained by cutting $X$ along the components of $\alpha$. The fiber $\Psi^{-1}(\mathbf{X},\mathbf{L})$ above any pair $(\mathbf{L}, \mathbf{X}) \in \Omega_{g}(\gamma)$ is given by all possible ways of glueing the components of $\smash{\mathbf{X}:=(X_j)_{j=1}^c}$ along their boundaries respecting the labelings. Given any point $(X,\alpha) \in \Psi^{-1}(\mathbf{X},\mathbf{L})$, the whole fiber $\Psi^{-1}(\mathbf{X},\mathbf{L})$ can be recovered by considering all possible Fenchel-Nielsen twists of $X$ along the components of $\alpha$.

Consider a point $(\mathbf{L},\mathbf{X}) \in \Omega_{g}(\gamma)$ on the base of the cut and glue fibration. Let $(X,\alpha) \in \Psi^{-1}(\mathbf{X},\mathbf{L})$ with $\alpha := (\alpha_i)_{i=1}^k$. The fiber $\Psi^{-1}(\mathbf{X},\mathbf{L})$ supports well defined $1$-forms $d\tau_{\alpha_i}$ that measure the infinitesimal Fenchel-Nielsen twists along the components of $\alpha$. Wedging these $1$-forms yields a natural volume form on $\Psi^{-1}(\mathbf{X},\mathbf{L})$. For every set of lengths  $\mathbf{L} := (\ell_i)_{i=1}^k$ and for Weil-Petersson almost every $\mathbf{X}$, the total mass of the fiber $\Psi^{-1}(\mathbf{X},\mathbf{L})$ with respect to this volume form is given by
\begin{equation}
\label{eq:d1}
| \Psi^{-1}(\mathbf{X},\mathbf{L}) | = 2^{-\rho_{g,n}(\gamma)} \cdot \ell_1 \cdots \ell_k,
\end{equation}
where $\rho_{g}(\gamma)$ is the number of components of $\gamma$ that bound (on any of its sides) a component of $S_{g}(\gamma)$ that is homeomorphic to a torus with one boundary component. We refer to pairs $(\mathbf{L},\mathbf{X}) \in \Omega_{g}(\gamma)$ such that $(\ref{eq:d1})$ holds as generic pairs. The factor $2^{-\rho_{g}(\gamma)}$ reflects the observation that every hyperbolic torus with one geodesic boundary component has a non-trivial isometric involution preserving its boundary setwise. See Figure \ref{fig:inv} for a picture of this involution.

\begin{figure}[h!]
	\centering
	\includegraphics[width=.3\textwidth]{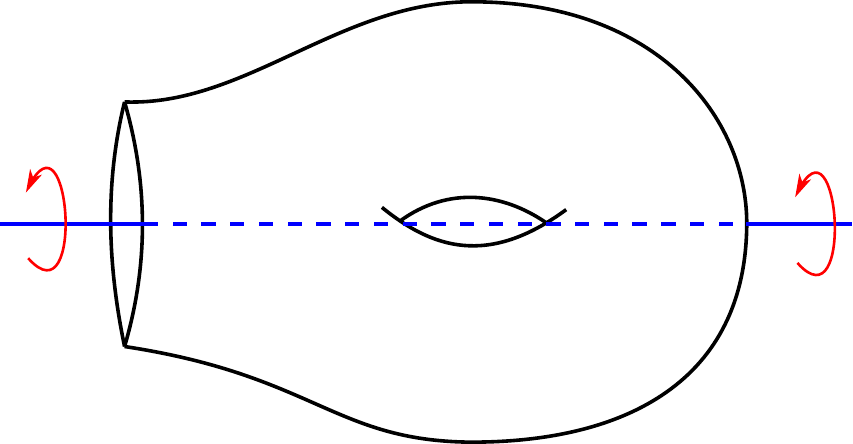}
	\caption{The isometric involution of a torus with one boundary component.} 
	\label{fig:inv} 
\end{figure}

Denote by $\mu_{\mathrm{wp}}$ the Weil-Petersson measure on $\mathcal{T}_{g}$ and by $\widetilde{\mu}_{\mathrm{wp}}$ the local pushforward of $\mu_{\mathrm{wp}}$ to $\mathcal{T}_g/\mathrm{Stab}_0(\gamma)$. It follows from Wolpert's magic formula, i.e., from (\ref{eq:wol1}) and (\ref{eq:wol2}), that $\widetilde{\mu}_{\mathrm{wp}}$ can be disintegrated along the fibers of the cut and glue fibration $\Psi \colon \mathcal{T}_{g,n} /\text{Stab}_0(\gamma) \to \Omega_{g,n}(\gamma)$ as
\[
\widetilde{\mu}_\mathrm{wp} = \sigma_g(\gamma) \cdot \prod_{i=1}^k d\tau_{\alpha_i} \otimes \prod_{i=1}^c \widehat{\mu}_{\mathrm{wp}}^j \otimes \prod_{i=1}^k d\ell_i,
\]
where $\widehat{\mu}_\text{wp}^j$ denotes the Weil-Petersson measure on $\mathcal{M}_{g_j,b_j}(\mathbf{L}_j)$ and $\sigma_g(\gamma) \colon \mathcal{T}_g/\mathrm{Stab}_0(\gamma) \to \mathbf{Q}_{>0}$ is the function with non-negative rational values that records the discrepancy between the stabilizer factors, i.e., the factors $|\Gamma|$ in (\ref{eq:stab_fact}), of the measure $\widetilde{\mu}_{\mathrm{wp}}$ and the product of measures $\prod_{i=1}^c \widehat{\mu}_{\mathrm{wp}}^j$. One can show that the function $\sigma_g(\gamma)$ is almost everywhere constant with respect to $\widetilde{\mu}_{\mathrm{wp}}$ and is given by
\[
\sigma_{g}(\gamma) := \frac{\prod_{j=1}^c |K_{g_j,b_j}|}{|\text{Stab}_0(\gamma)\cap K_{g}|},
\]
where $K_{g_j,b_j} \triangleleft \text{Mod}_{g_j,b_j}$ is the kernel of the mapping class group action on $\mathcal{T}_{g_j,b_j}$ and $K_{g} \triangleleft \text{Mod}_{g}$ is the kernel of the mapping class group action on $\mathcal{T}_{g}$. Indeed, $\widehat{\mu}_{\mathrm{wp}}^j$ almost every hyperbolic surface in $\mathcal{M}_{g_j,b_j}(\mathbf{L})$ has $|K_{g_j,b_j}|$ automorphisms and analogously for elements of $\mathcal{T}_g/\mathrm{Stab}_0(\gamma)$. For example, if $g =2$ and $\gamma$ is a separating simple closed curve on $S_{2}$, then $\sigma_{2}(\gamma) = 4/2 = 2$. 

Let us record the most important conclusions of the discussion above in the following theorem.

\begin{theorem}
	\label{theo:cut_glue_fib}
	The Weil-Petersson measure $\widetilde{\mu}_{\mathrm{wp}}$ on $\mathcal{T}_{g} /\text{Stab}_0(\gamma)$ can be disintegrated along the fibers of the cut and glue fibration $\Psi \colon \mathcal{T}_{g,n} /\text{Stab}_0(\gamma) \to \Omega_{g,n}(\gamma)$ in the following way,
	\[
	\widetilde{\mu}_\mathrm{wp} = \sigma_g(\gamma) \cdot \prod_{i=1}^k d\tau_{\alpha_i} \otimes \prod_{i=1}^c \widehat{\mu}_{\mathrm{wp}}^j \otimes \prod_{i=1}^k d\ell_i.
	\]
	Generically, the volume of the fiber $\Psi^{-1}(\mathbf{X},\mathbf{L})$ above a $(\mathbf{L},\mathbf{X}) \in \Omega_{g,n}(\gamma)$ with $\mathbf{L} := (\ell_i)_{i=1}^k$ is equal to
	\[
	| \Psi^{-1}(\mathbf{X},\mathbf{L}) | = 2^{-\rho_{g}(\gamma)} \cdot \ell_1 \cdots \ell_k.
	\]
\end{theorem}

\subsection*{Mirzakhani's integration formula.} We are now ready to discuss Mirzakhani's integration formula. Let $\gamma := (\gamma_i)_{i=1}^k$ be an ordered simple closed multi-curve on $S_g$. Consider the stabilizer
\[
\mathrm{Stab}(\gamma) := \bigcap_{i=1}^k \mathrm{Stab}(\gamma_i) \subseteq \mcg.
\]
Notice that $\mathrm{Stab}_0(\gamma)$ is a finite index subgroup of $\mathrm{Stab}(\gamma)$. Let $c > 0$ and $g_j,b_j \geq 0$ with $j \in \{1, \dots, c\}$ be as above. Consider the total Weil-Petersson volumes
\[
V_{g_j,b_j}(\mathbf{x}_j) := \widehat{\mu}^j_{\mathrm{wp}}(\mathcal{M}_{g_j,b_j}(\mathbf{x}_j))
\]
of the moduli spaces $\mathcal{M}_{g_j,b_j}(\mathbf{x}_j)$ as functions of the boundary lengths $\mathbf{x}_j \in \mathbf{R}_+^{b_j}$. Recall the definitions of the constants $\rho_{g}(\gamma) \in \mathbf{N}$ and $\sigma_g(\gamma) \in \mathbf{Q}_{>0}$ introduced above. Consider the measurable function $V_g(\gamma,\cdot) \colon \mathbf{R}_+^k \to \mathbf{R}^+$ which to every vector $\mathbf{x} := (x_i)_{i=1}^k$ with positive entries assigns the value
\begin{equation}
\label{eq:V}
V_g(\gamma,\mathbf{x}) := \frac{\sigma_g(\gamma) \cdot 2^{-\rho_g(\gamma)}}{[\mathrm{Stab}(\gamma):\mathrm{Stab}_0(\gamma)]}  \cdot \prod_{j=1}^c V_{g_j,b_j}(\mathbf{x}_j),
\end{equation}
where the vector $\mathbf{x}_j \in \mathbf{R}_+^{b_j}$ is defined using the correspondence between the components of $\gamma$ and the boundary components of $S_{g_j,b_j}$. By Theorem \ref{theo:vol_pol}, $V_g(\gamma,\mathbf{x})$ is a polynomial on the $\mathbf{x}$ variables.

Let $f \colon \smash{\mathbf{R}_+^k} \to \mathbf{R}$ be a non-negative, measurable function. Given a marked hyperbolic structure $X \in \mathcal{T}_g$ and an ordered simple closed multi-curve $\alpha$ on $S_g$, denote $\smash{\vec{\ell}_\alpha}(X) := (\ell_{\alpha_i}(X))_{i=1}^k \in \mathbf{R}_+^k$. Consider the non-negative, measurable function $f_\gamma \colon \mathcal{T}_g \to \mathbf{R}$ which to every $X \in \mathcal{T}_g$ assigns the value
\[
f_\gamma(X) := \sum_{\alpha \in \mcg \cdot \gamma} f\left(\smash{\vec{\ell}_\alpha}(X)\right).
\]
This function is clearly invariant with respect to the action of $\mcg$ on $\mathcal{T}_g$. Denote by $\smash{\widehat{f}}_\gamma \colon \mathcal{M}_g \to \mathbf{R}$ the corresponding non-negative, measurable function induced on moduli space. Denote by $\widehat{\mu}_{\mathrm{wp}}$ the Weil-Petersson measure on $\mathcal{M}_g$. On $\smash{\mathbf{R}_+^k}$ consider the standard coordinate system $\mathbf{x} := \smash{(x_i)_{i=1}^k}$ and the Lebesgue class measure $\mathbf{x} \cdot d\mathbf{x} := x_1 \cdots x_k \cdot dx_1 \cdots dx_k$. The following exercise corresponds to Mirzakhani's integration formula \cite[Theorem 4.1]{Mir08b}.

\begin{exercise}
	\label{ex:mir_int}
	Using Exercise \ref{ex:local_cv} and Theorem \ref{theo:cut_glue_fib} show that the following integration formula holds,
	\[
	\int_{\mathcal{M}_g} \widehat{f}_\gamma(X) \thinspace d\widehat{\mu}_{\mathrm{wp}}(X) = \int_{\mathbf{R}_+^k} f(\mathbf{x}) \cdot V_g(\gamma,\mathbf{x}) \cdot \mathbf{x} \cdot d\mathbf{x}.
	\]
\end{exercise}

\subsection*{McShane's identity.} We now sketch a proof of Theorem \ref{theo:vol_pol} for the case of tori with one boundary component, i.e., for the case $g=1$ and $b=1$. For the rest of this discussion we fix a torus with one boundary component $S_{1,1}$ and a parameter $L > 0$. To compute the total Weil-Petersson volume $V_{1,1}(L) := \widehat{\mu}_\mathrm{wp}(\mathcal{M}_{1,1}(L))$ one could recall $\mathcal{M}_{1,1}(L) := \mathcal{T}_{1,1}(L)/ \mathrm{Mod}_{1,1}$ and try to find a fundamental domain for the action of $\mathrm{Mod}_{1,1}$ on $\mathcal{T}_{1,1}(L)$ that can be described explicitely in Fenchel-Nielsen coordinates. This happens to be quite a formidable task. In abscence of such a fundamental domain we consider a different approach that relies on Mirzakhani's integration formula.

For the rest of this discussion we fix a simple closed curve $\gamma$ on $S_{1,1}$.  Let $f \colon \mathbf{R}^+ \to \mathbf{R}$ be a non-negative measurable function. Just as in the case of closed surfaces, consider the non-negative, measurable transform $\smash{\widehat{f}_\gamma} \colon \mathcal{M}_{1,1}(L) \to \mathbf{R}$ which to every $X \in \mathcal{M}_{1,1}(L)$ assigns the value
\[
\smash{\widehat{f}_\gamma}(X) = \sum_{\alpha \in \mathrm{Mod}_{1,1} \cdot \gamma} f(\ell_{\gamma}(X)).
\]
An analogue of Mirzakhani's integration formula in Exercise \ref{ex:mir_int} also holds in this setting. Our immediate goal is to find a non-negative, measurable function $f \colon \mathbf{R}^+ \to \mathbf{R}$ for which the transform $\smash{\widehat{f}_\gamma} \colon \mathcal{M}_{1,1}(L) \to \mathbf{R}$ is equal to a constant $c(L) > 0$. Indeed, for such a function we would have
\begin{equation}
\label{eq:int}
c(L) \cdot \widehat{\mu}_\mathrm{wp}(\mathcal{M}_{1,1}(L)) = \int_{\mathcal{M}_{1,1}(L)} \widehat{f}_\gamma(X) \thinspace d\widehat{\mu}_\mathrm{wp}(X) = \frac{1}{2} \cdot \int_{\mathbf{R}^+}f(x) \cdot x \cdot dx.
\end{equation}
Rearranging the terms in this equation would yield
\begin{equation}
\label{eq:trick}
V_{1,1}(L) := \widehat{\mu}_\mathrm{wp}(\mathcal{M}_{1,1}(L)) = \frac{1}{2 \cdot c(L)} \cdot \int_{\mathbf{R}^+}f(x) \cdot x \cdot dx.
\end{equation}

Finding a non-negative, measurable function $f \colon \mathbf{R}^+ \to \mathbf{R}$ for which $\smash{\widehat{f}_\gamma} \colon \mathcal{M}_{1,1}(L) \to \mathbf{R}$ is constant is the content of McShane's identity. Consider the function $D \colon \mathbf{R}^3 \to \mathbf{R}$ given by
\[
D(x,y,z) := 2 \log \left(\frac{e^{\frac{x}{2}} + e^{\frac{y+z}{2}}}{e^{-\frac{x}{2}} + e^{\frac{y+z}{2}}} \right).
\]

\begin{theorem}
	\cite{Mc91}
	\label{theo:mc_shane}
	Let $\gamma$ a simple closed curve on $S_{1,1}$ and $L > 0$. Then, for every $X \in \mathcal{M}_{1,1}(L)$,
	\begin{equation}
	\label{eq:mc}
	\sum_{\alpha \in \mathrm{Mod}_{1,1} \cdot \gamma} D(L,\ell_{\alpha}(X),\ell_{\alpha}(X)) = L.
	\end{equation}
\end{theorem}

Let us give a quick rundown of the main ideas of the proof of Theorem \ref{theo:mc_shane}. For this we interpret the right hand side of (\ref{eq:mc}) as the length of the boundary component of $X \in \mathcal{M}_{1,1}(L)$. For every point on this boundary component we shoot a geodesic into $X$ in the direction orthogonal to the boundary. Three things can happen at this stage. Either the geodesic remains in $X$ at all times without intersecting itself, the geodesic exits $X$ without intersecting itself, or the geodesic intersects itself. By work of Birman and Series \cite{BS85}, the first case only happens for a measure zero subset of points on the boundary. In each of the two other cases we consider the simple closed curves $\alpha$ described in Figure \ref{fig:mc}. Cutting $X$ along the corresponding geodesic representatives yields a hyperbolic pair of pants with boundary lengths $(L,\ell_\alpha(X),\ell_{\alpha}(X))$. Using the rigidity of such pairs of pants one can show that for points in exactly two arcs of the original boundary component whose lengths add up to $D((L,\ell_\alpha(X),\ell_{\alpha}(X))$, the corresponding orthogonal geodesics either intersect themselves or the original boundary component before exiting the pants. Putting these ideas together finishes the proof.

\begin{figure}[h]
	\centering
	\begin{subfigure}[b]{0.4\textwidth}
		\centering
		\includegraphics[width=.6\textwidth]{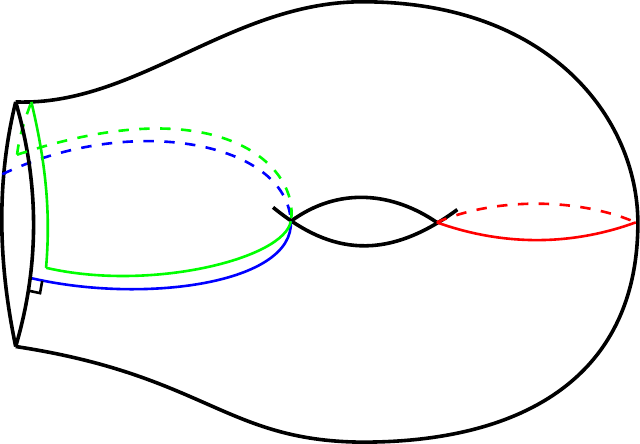}
		\caption{Geodesic intersects boundary first.}
	\end{subfigure}
	\quad \quad \quad
	~ 
	\begin{subfigure}[b]{0.4\textwidth}
		\centering
		\includegraphics[width=.6\textwidth]{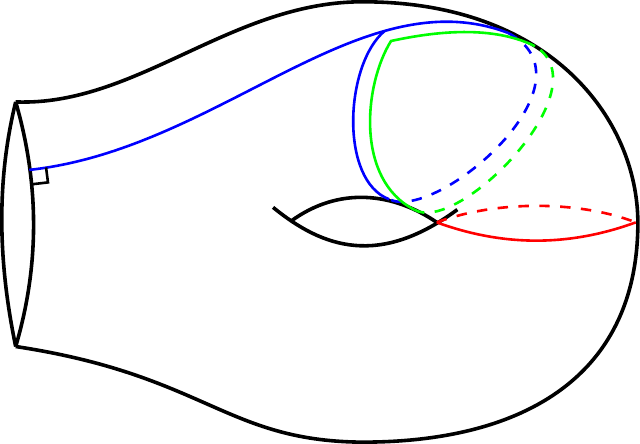}
		\caption{Geodesic intersects itself first.}
	\end{subfigure}
	\caption{The geodesic shot orthogonal to the boundary (in blue) determines a simple closed curve (in green) which tightens to a simple closed geodesic (in red).} 
	\label{fig:mc}
\end{figure}

\begin{exercise}
	Using (\ref{eq:trick}) and Theorem \ref{theo:mc_shane} show that
	\begin{equation}
	\label{eq:f1}
	V_{1,1}(L) = \frac{1}{48} \cdot L^2 + \frac{\pi^2}{12}.
	\end{equation}
	\textit{Hint: To simplify the computations, differentiate (\ref{eq:int}) with respect to $L$ instead of using (\ref{eq:trick}) directly.}
\end{exercise}

In her thesis \cite{Mir04}, Mirzakhani generalized Theorem \ref{theo:mc_shane} to arbitrary closed, orientable surfaces \cite[Theorem 1.3]{Mir07a}. Using this generalization, Mirzakhani proved a recursive formula for the total Weil-Petersson volumes $V_{g,b}(\mathbf{L})$  \cite[\S 5]{Mir07a}. Using this formula she ultimately deduced Theorem \ref{theo:vol_pol}. An alternative proof of Theorem \ref{theo:vol_pol} using symplectic reduction can also be found in Mirzakhani's thesis \cite[Theorem 1.1]{Mir07c}. An excellent reference for all these topics is Do's survey \cite{Do13}.

In subsequent discussions we will use the following explicit volume polynomial,
\begin{equation}
\label{eq:f2}
V_{1,2}(L_1,L_2) = \frac{1}{192} \cdot L_1^4 + \frac{1}{96} \cdot L_1^2  L_2^2 + \frac{1}{192} \cdot L_2^4 + \frac{\pi^2}{12} \cdot L_1^2 + \frac{\pi^2}{12} \cdot L_2^2 + \frac{\pi^2}{4}.
\end{equation}

\begin{exercise}
	\label{ex:VV}
	Recall the definition of the polynomials $V_g(\gamma,\mathbf{x})$ in (\ref{eq:V}). Let $\gamma_1$ and $\gamma_2$ be non-separating and separating simple closed curves on $S_2$, respectively. Using (\ref{eq:f1}) and (\ref{eq:f2}) show that
	\begin{align*}
	V_2(\gamma_1,x) = \frac{1}{96} \cdot x^4, \quad
	V_2(\gamma_2,x) = \frac{1}{4608} \cdot x^4.
	\end{align*}
\end{exercise}

\section{Counting simple closed geodesics on hyperbolic surfaces}

\subsection*{Outline of this section.} In this section we give a complete proof of Theorem \ref{theo:main}, the main result of this survey. We follow the same approach as in the proof of Theorem \ref{theo:main_lec_1} but use the vocabulary and tools introduced in \S 3 -- 4. We encourage the reader to keep in mind the analogies described in Table \ref{tb:1} for the rest of this section.

\subsection*{Counting simple closed geodesics.} Recall that Theorem \ref{theo:main}, the main result of this survey, asserts that the number $s(X,L)$ of unoriented simple closed geodesics of length $\leq L$ on an arbitrary closed, orientable hyperbolic surface $X$ of genus $g \geq 2$ is asymptotic as $L \to \infty$ to a polynomial of degree $6g-6$ on $L$. To prove this result we follow the same approach considered in the proof of Theorem \ref{theo:main_lec_1_red}. We begin by rewriting the counting function of interest $s(X,L)$ using the vocabulary introduced in previous sections. 

For the rest of this section we fix an integer $g \geq 2$ and a connected, oriented, closed surface $S_g$ of genus $g$. Recall that a simple closed curve on $S_g$ is an equivalence class of homotopically non-trivial parametrized simple closed curves on $S_g$ up to free homotopy and orientation reversal. Recall that $\mcg$ denotes the mapping class group of $S_g$ and that this group acts naturally on the set of simple closed curves on $S_g$. Recall that orbits of this action are in one-to-one correspondence with topological types of simple closed curves on $S_g$. The following exercise will be crucial for our approach.

\begin{exercise}
	\label{ex:finite_top_type}
	Show there are only finitely many topological types of simple closed curves on $S_g$. Can you give an exact formula for the number of such equivalence classes?
\end{exercise}

Recall that $\mathcal{T}_g$ denotes the Teichmüller space of marked hyperbolic structures on $S_g$. Recall that if $\gamma$ is a simple closed curve on $S_g$ and $X \in \mathcal{T}_g$ then $\ell_{\gamma}(X) > 0 $ denotes the length of the unique geodesic representative of $\gamma$ with respect to $X$. Let $\gamma$ be a simple closed curve on $S_g$ and $X \in \mathcal{T}_g$ be a marked hyperbolic structure on $S_g$. For every $L > 0$ consider the counting function
\[
s(X,\gamma,L) := \# \{\alpha \in \mcg \cdot \gamma \ | \ \ell_{\alpha}(X) \leq L \}.
\]
As every free homotopy class of simple closed curves on $S_g$ has a unique geodesic representative with respect to $X$, this is exactly the number of unoriented simple closed geodesics on $X$ of the same topological type as $\gamma$ and length $\leq L$. Hence, we can decompose the counting function $s(X,L)$ as
\begin{equation}
\label{eq:sum}
s(X,L) = \sum_{\gamma} s(X,\gamma,L),
\end{equation}
where $\gamma$ runs over all the finitely many different topological types of simple closed curves on $S_g$. See Exercise \ref{ex:finite_top_type}. Thus, it is enough for our purposes to study the asymptotics of $s(X,\gamma,L)$ as $L \to \infty$.

Motivated by this observation we generalize the definition of the counting function $s(X,\gamma,L)$ to general integral simple closed multi-curves on $S_g$. Recall that if $\alpha := \smash{\sum_{i=1}^k a_i \alpha_i}$ is a simple closed multi-curve on $S_g$ and $X \in \mathcal{T}_g$ then $\ell_{\alpha}(X) := \sum_{i=1} a_i \ell_{\alpha_i}(X) > 0$. Let $\gamma := \smash{\sum_{i=1}^k a_i \gamma_i}$ be an integral simple closed multi-curve on $S_g$ and $X \in \mathcal{T}_g$. For every $L > 0$ consider the counting function
\[
s(X,\gamma,L) := \#\{\alpha \in \mcg \cdot \gamma \ | \ \ell_{\alpha}(X) \leq L \}.
\]

Recall that $\mathcal{M}_g := \mathcal{T}_g/\mcg$ denotes the moduli space of hyperbolic structures on $S_g$. Notice that the counting function $s(X,\gamma,L)$ does not depend on the marking of $X \in \mathcal{T}_g$ but only on its underlying hyperbolic structure $X \in \mathcal{M}_g$. We aim to prove the following asymptotic estimate for the counting function $s(X,\gamma,L)$. Theorem \ref{theo:main} will later be deduced as a direct consequence of this estimate. 

\begin{theorem}
	\label{theo:main_5}
	Let $\gamma := \smash{\sum_{i=1}^k a_i \gamma_i}$ be an integral simple closed multi-curve on $S_g$ and $X \in \mathcal{M}_g$. Then, there exists a constant $n(X,\gamma) > 0$ depending only on the topological type of $\gamma$ and the geometry of $X$ such that the following asymptotic estimate holds as $L \to \infty$,
	\[
	s(X,\gamma,L) \sim n(X,\gamma) \cdot L^{6g-6}.
	\]
\end{theorem}

\subsection*{Counting measures on the space of measured geodesic laminations.} Inspired by the case of lattices, to prove Theorem \ref{theo:main_5} we introduce appropriate families of counting measures. Recall that $\ml$ denotes the space of measured geodesic laminations on $S_g$, that is, the natural $6g-6$ dimensional completion of the set of simple closed multi-curves on $S_g$. Let $\gamma := \smash{\sum_{i=1}^k a_i \gamma_i}$ be an integral simple closed multi-curve on $S_g$. For every $L > 0$ consider the counting measure on $\mathcal{ML}_g$ given by
\[
\mu^\gamma_L := \frac{1}{L^{6g-6}} \cdot \sum_{\alpha \in \mcg \cdot \gamma} \delta_{\frac{1}{L} \cdot \alpha}.
\]
Recall that the length $\ell_\lambda(X) > 0$ of a measured geodesic lamination $\lambda \in \ml$ with respect to a marked hyperbolic structure $X \in \mathcal{T}_g$ can be defined in a unique continuous way extending the definition on simple closed multi-curves. Let $X \in \mathcal{T}_g$ be a marked hyperbolic structure on $S_g$. Consider the subset
\[
B_X : = \{\lambda \in \ml \ | \ \ell_\lambda(X) \leq 1\}.
\]
A direct computation shows that for every $L > 0$,
\begin{equation}
\label{eq:connection}
\mu^\gamma_L\left(B_X\right) = \frac{s(X,\gamma,L)}{L^{6g-6}}.
\end{equation}
This reduces the original problem of proving an asymptotic estimate for the counting function $s(X,\gamma,L)$ to the problem of understanding the behavior as $L \to \infty$ of the sequence of counting measures $(\smash{\mu^\gamma_L})_{L > 0}$.

\subsection*{Ergodicity of the mapping class group action.} To study the asymptotic behavior of the sequence of counting measures $(\smash{\mu^\gamma_L})_{L > 0}$ we use the dynamics of the action of the mapping class group on $\ml$. Recall the definition of the Thurston measure $\mu_{\mathrm{Thu}}$ on $\ml$ and its relation to Dehn-Thurston coordinates as described in Exercise \ref{ex:thu_1}. The following result of Masur is an analogue of the fact that the Lebesgue measure is ergodic with respect to the action of $\mathrm{SL}(2,\mathbf{Z})$ on $\mathbf{R}^2$. See Theorem \ref{theo:sl_ergodic}.

\begin{theorem}
	\cite{Mas85}
	\label{theo:thu_erg}
	The measure $\mu_\mathrm{Thu}$ is ergodic with respect to the action of $\mcg$ on $\ml$.
\end{theorem}

In analogy with the case of lattices, the ergodicity of the mapping class group action on $\ml$ can be used to study the weak-$\star$ limit points of the sequence of counting measures $(\smash{\mu_L^\gamma})_{L>0}$. The following crucial exercise is an analogue of Proposition \ref{prop:prim_meas_1}.

\begin{exercise}
	\label{ex:prim_meas_2}
	Let $\gamma := \sum_{i=1}^k a_i \gamma_i$ be an integral simple closed multi-curve on $S_g$. Show that every weak-$\star$ limit point $\mu^\gamma$ of the sequence of counting measures $(\smash{\mu_L^\gamma})_{L>0}$ is of the form $\mu^\gamma = c \cdot \mu_{\mathrm{Thu}}$ for some constant $c\geq 0$. \textit{Hint: Follow the same approach as in the proof of Proposition \ref{prop:prim_meas_1}. To prove $\mu^\gamma$ is absolutely continuous with respect to $\mu_{\mathrm{Thu}}$ use Dehn-Thurston coordinates and Exercise \ref{ex:thu_1}.}
\end{exercise}

\subsection*{Integration and integrability.} Our next goal is to show that the constant $c \geq 0$ in the conclusion of Exercise \ref{ex:prim_meas_2} is positive and independent of the limit point $\mu^\gamma$. Similar to the case of lattices, we achieve this goal by averaging over moduli space and using Mirzakhani's integration formula.

Recall that $\widehat{\mu}_{\mathrm{wp}}$ denotes the Weil-Petersson measure on $\mathcal{M}_g$. Let $\gamma := (\gamma_i)_{i=1}^k$ be an ordered simple closed multi-curve on $S_g$ and $\mathbf{a} := \smash{(a_i)_{i=1}^k} \in \mathbf{R}_+^k$ be a vector of positive weights. On $S_g$ consider the simple closed multi-curve  given by $\mathbf{a} \cdot \gamma := \smash{\sum_{i=1}^k a_i \gamma_i}$. Recall the definition of the polynomial $V_g(\gamma,\mathbf{x})$ introduced in (\ref{eq:V}). Notice that $\mathrm{Stab}(\gamma) \subseteq \mcg$ is a finite index subgroup of $\mathrm{Stab}(\mathbf{a} \cdot \gamma)$. Define
\[
V_g(\mathbf{a} \cdot \gamma, \mathbf{x}) := [\mathrm{Stab}(\mathbf{a} \cdot \gamma):\mathrm{Stab}(\gamma)]^{-1} \cdot V_g(\gamma,\mathbf{x}).
\]
On $\smash{\mathbf{R}_+^k}$ consider the standard coordinate system $\mathbf{x} := \smash{(x_i)_{i=1}^k}$ and the Lebesgue class measure $\mathbf{x} \cdot d\mathbf{x} := x_1 \cdots x_k \cdot dx_1 \cdots dx_k$. The following exercise is an analogue of Proposition \ref{prop:siegel_basic} and Exercise \ref{ex:long}.

\begin{exercise}
	\label{ex:freq}
	Let $\gamma :=(\gamma_i)_{i=1}^k$ be an ordered simple closed multi-curve on $S_g$ and $\mathbf{a}:= (a_i)_{i=1}^k \in \mathbf{R}_+^k$ be a vector of positive weights on the components of $\gamma$. Using Exercise $\ref{ex:mir_int}$ show that for every $L > 0$,
	\[
	\int_{\mathcal{M}_g} s(X,\mathbf{a} \cdot \gamma,L) \thinspace d\widehat{\mu}_\mathrm{wp}(X) = \int_{\mathbf{a} \cdot \mathbf{x} \leq L} V_g(\mathbf{a} \cdot \gamma,\mathbf{x}) \cdot \mathbf{x} \cdot d\mathbf{x}.
	\]
	Using this formula and Theorem \ref{theo:vol_pol} deduce that the function
	\begin{equation}
	\label{eq:P}
	P(\mathbf{a} \cdot \gamma,L) := \int_{\mathcal{M}_g} s(X,\mathbf{a} \cdot \gamma,L) \thinspace d\widehat{\mu}_\mathrm{wp}(X) 
	\end{equation}
	is a polynomial of degree $6g-6$ on $L$ with rational leading coefficient.
\end{exercise}

Let $\gamma:= \smash{\sum_{i=1}^k a_i \gamma_i}$ be a simple closed multi-curve on $S_g$ and $P(\gamma,L)$ be as in (\ref{eq:P}). Following Exercise \ref{ex:freq} we define the frequency of $\gamma$ to be the positive rational number
\begin{equation}
\label{eq:freq}
c(\gamma) := \lim_{L \to \infty} \frac{P(\gamma,L)}{L^{6g-6}}.
\end{equation}
Using Mirzakhani's recursion for the Weil-Petersson volume polynomials $V_{g,b}(\mathbf{L})$ \cite[\S 5]{Mir07a}, the frequency of any simple closed curve can be computed explicitely by means of a recursive algorithm.

\begin{exercise}
	\label{ex:freq_2}
	Let $\gamma_1$ and $\gamma_2$ be non-separating and separating simple closed curves  on $S_2$, respectively. Using Exercise \ref{ex:VV} show that the frequencies of these simple closed curves are given by
	\[
	c(\gamma_1) =  1/576, \quad c(\gamma_2) = 1/27648.
	\]
\end{exercise}

We now discuss an analogue of the integrability bound in Exercise \ref{eq:dct} for the function $s(X,\gamma,L)$. Let us first introduce a more precise statement of the collar lemma discussed in \S 3. A proof of this result can be found in \cite[Lemma 13.6]{FM11}. Consider the width function $w \colon \mathbf{R}^+ \to \mathbf{R}^+$ given by
\[
w(x) := \mathrm{arcsinh}\left( \frac{1}{\sinh \left(\frac x 2 \right)}\right).
\]

\begin{lemma}
	\label{lem:collar_lemma}
	Let $\gamma$ be a simple closed geodesic on a closed, orientable hyperbolic surface $X$. Denote by $d$ the metric on $X$. Then, the subset $N_{\gamma}  \subseteq X$ defined as follows is an embedded annulus in $X$,
	\[
	N_{\gamma} := \{x \in X  \colon  d(x,\gamma) < w(\ell_\gamma(X))\}.
	\] 
\end{lemma}

Let $\mathcal{P} := \smash{(\gamma_i)_{i=1}^{3g-3}}$ be a pair of pants decomposition of $S_g$ and $\smash{(m_i,t_i)_{i=1}^{3g-3}}$ be a set of Dehn-Thurston coordinates of $\ml$ induced by $\mathcal{P}$. Given a measured geodesic lamination $\lambda \in \ml$ and $X \in \mathcal{T}_g$, define the combinatorial length of $\lambda$ with respect $X$ and $\mathcal{P}$ as
\[
L_\lambda(X,\mathcal{P}) := \sum_{i=1}^N (m_i(\gamma) \cdot w(\ell_{\gamma_i}(X)) + |t_i(\gamma)| \cdot \ell_{\gamma_i}(X)).
\]
For a simple closed curve $\alpha$ on $S_g$ and $X \in \mathcal{T}_g$ this definition has a concrete interpretation: add the width $w(\ell_{\gamma_i}(X))$ of the collar given by Lemma \ref{lem:collar_lemma} to the combinatorial length of $\alpha$ every time it intersects $\gamma_i$ and add the length $\ell_{\gamma_i}(X)$ to the combinatorial length of $\alpha$ every time it twists around $\gamma_i$. A pair of pants decomposition $\mathcal{P} := \smash{(\gamma_i)_{i=1}^{3g-3}}$ of $S_g$ is said to be $L$-bounded with respect to $X \in \mathcal{T}_g$ for some $L > 0$ if $\ell_{\gamma_i}(X) \leq L$ for every $i \in \{1,\dots,3g-3\}$. The following result of Mirzakhani shows that combinatorial lengths approximate hyperbolic lengths in a coarse sense.

\begin{proposition}
	\cite[Proposition 3.5]{Mir08b}
	\label{prop:ength_comp}
	For every $L > 0$ there exists $C = C(L) > 0$ such that for every $X \in \mathcal{T}_{g}$ and every pair of pants decomposition $\mathcal{P}$ of $S_{g}$ that is $L$-bounded with respect to $X$ there exist Dehn-Thurston coordinates $\smash{(m_i,t_i)_{i=1}^{3g-3}}$ of $\mathcal{ML}_{g}$ induced by $\mathcal{P}$ such that for every $\lambda \in \mathcal{ML}_{g}$,
	\[
	C^{-1} \cdot L_{\mathcal{P}}(X,\lambda) \leq  \ell_\lambda(X) \leq C \cdot L_{\mathcal{P}}(X,\lambda).
	\]
\end{proposition}

To prove the aforementioned integrability bound for the counting function $s(X,\gamma,L)$ we will also use the following stronger version of Bers's theorem. For a proof see
\cite[Theorem 12.8]{FM11}.

\begin{theorem}
	\label{theo:bers}
	For every $\epsilon > 0$ there exists $L = L(\epsilon) > 0$ with the following property. Let $X \in \mathcal{T}_{g}$ be a marked hyperbolic structure and $\gamma := (\gamma_i)_{i=1}^k$ be a simple closed multi-curve on $S_g$ such that
	\[
	\ell_{\gamma_i}(X) < \epsilon, \ \forall i=1,\dots,k.
	\]
	Then, there exists a completion of $\gamma$ to a pair of pants decomposition $\mathcal{P}:= (\gamma_i)_{i=1}^{3g-3}$ of $S_g$ such that
	\[
	\ell_{\gamma_i}(X) < L, \ \forall i=1,\dots,3g-3.
	\]
\end{theorem}
$ $

By Lemma \ref{lem:collar_lemma}, there exists a constant $\epsilon > 0$ such that on any closed, orientable hyperbolic surface no two closed geodesics of length $< \epsilon$ intersect. For the rest of this section we fix such a constant and denote it by $\epsilon > 0$. Consider the measurable function $u \colon \mathcal{M}_g \to \mathbf{R}^+$ given for every $X \in \mathcal{M}_g$ by
\[
u(X) := \prod_{\gamma \colon \ell_{\gamma}(X) < \epsilon} \frac{1}{\ell_{\gamma}(X)},
\]
where the product runs over all simple closed geodesics $\gamma$ on $X$ of length $\ell_{\gamma}(X) < \epsilon$. We interpret empty products as taking the value $1$. The following exercise is an analogue of Exercise \ref{eq:dct}.

\begin{exercise}
	\label{ex:int_strong}
	Show there exists a constant $C > 0$ such that for every $X \in \mathcal{M}_g$ and every $L > 0$,
	\begin{equation}
	\label{eq:V1}
	s(X,\gamma,L) \leq C \cdot L^{6g-6} \cdot u(X).
	\end{equation}
	Additionally, show that the function $u \colon \mathcal{M}_g \to \mathbf{R}^+$ is integrable with respect to $\widehat{\mu}_\mathrm{wp}$, i.e.,
	\begin{equation}
	\label{eq:V2}
	\int_{\mathcal{M}_g} u(X) \thinspace d\widehat{\mu}_\mathrm{wp}(X) < \infty.
	\end{equation}
	\textit{Hint: To prove (\ref{eq:V1}) use Proposition \ref{prop:ength_comp} and Theorem \ref{theo:bers} to reduce to a lattice point counting problem in Dehn-Thurston coordinates. To prove (\ref{eq:V2}) follow the same approach as in Exercise \ref{ex:wp_vol}.}
\end{exercise}

\subsection*{Equidistribution of counting measures.} We are finally ready to prove that the sequence of counting measure $\smash{(\nu^\gamma_L)_{L > 0}}$ on $\ml$ converges in the weak-$\star$ topology to a constant multiple of the Thurston measure $\mu_\mathrm{Thu}$. The following exercise is an analogue of Theorem \ref{theo:equid_lat}.

\begin{exercise}
	\label{ex:equid}
	Let $\gamma:= \smash{\sum_{i=1}^k a_i \gamma_i}$ be an integral simple closed multi-curve on $S_g$. Show that, with respect to the weak-$\star$ topology for measures on $\ml$,
	\[
	\lim_{L \to \infty} \mu_L^\gamma = c(\gamma) \cdot \mu_{\mathrm{Thu}}.
	\]
	\textit{Hint: Follow the same approach as in the proof of Theorem \ref{theo:equid_lat}.}
\end{exercise}

\subsection*{Counting simple closed multi-curves.} For every marked hyperbolic structure $X \in \mathcal{T}_g$ denote
\begin{equation}
\label{eq:B}
B(X):= \mu_{\mathrm{Thu}}\left( \{\lambda \in \ml \ | \ \ell_\lambda(X) \leq 1\}\right).
\end{equation}
As the Thurston measure $\mu_{\mathrm{Thu}}$ is invariant with respect to the $\mcg$ action on $\ml$, the value $B(X)$ is independent of the marking of $X \in \mathcal{T}_g$ and depends only on the underlying hyperbolic structure $X \in \mathcal{M}_g$. Thus, (\ref{eq:B}) gives rise to a function $B \colon \mathcal{M}_g \to \mathbf{R}^+$ known as the Mirzakahani function.

\begin{exercise}
	\label{ex:int}
	Show there exists a constant $C > 0$ such that for every $X \in \mathcal{M}_g$
	\begin{equation*}
	B(X) \leq C \cdot u(X).
	\end{equation*}
	Conclude that the function $B \colon \mathcal{M}_g \to \mathbf{R}^+$ is integrable with respect to $\widehat{\mu}_\mathrm{wp}$, i.e.,
	\[
	\int_{\mathcal{M}_g} B(X) \thinspace d \widehat{\mu}_\mathrm{wp}(X) < \infty.
	\]
	\textit{Hint: Use Proposition \ref{prop:ength_comp} and interpret $\mu_{\mathrm{Thu}}$ as the Lebesgue measure in Dehn-Thurston coordinates.}
\end{exercise}

Following Exercise \ref{ex:int} we consider the constant $b_g > 0$ defined as
\begin{equation}
\label{eq:bg}
b_g := \int_{\mathcal{M}_g} B(X) \thinspace d \widehat{\mu}_\mathrm{wp}(X).
\end{equation}

We are now ready to prove the following more precise version of Theorem \ref{theo:main_5}.

\begin{exercise}
	\label{ex:main_6}
	Let $\gamma:= \smash{\sum_{i=1}^k a_i \gamma_i}$ be an integral simple closed multi-curve on $S_g$ and $X \in \mathcal{T}_g$ be a marked hyperbolic structure on $S_g$. Show that the following identity holds
	\[
	\lim_{L \to \infty} \frac{s(X,\gamma,L)}{L^{6g-6}} = \frac{c(\gamma) \cdot B(X)}{b_g}.
	\]
	\textit{Hint: Follow the same approach as in the proof of Theorem \ref{theo:main_lec_1_red}. Aside from Exercises \ref{ex:equid} and \ref{ex:int}, it will be useful to recall the identity in (\ref{eq:connection}) as well as Exercise \ref{eq:boundary}. }
\end{exercise}

We are finally ready to prove Theorem \ref{theo:main}, the main result of this survey.

\begin{exercise}
	Using (\ref{eq:sum}) and Exercise \ref{ex:main_6} prove Theorem \ref{theo:main}, i.e., show that for every closed, orientable hyperbolic surface $X$ of genus $g \geq 2$ there exists a constant $n(X) > 0$ such that
	\[
	\lim_{L \to \infty} \frac{s(X,L)}{L^{6g-6}} = n(X).
	\]
\end{exercise}

Let us end this section by proving the following more precise version of Theorem \ref{theo:main_2}.

\begin{exercise}
	Let $\gamma_1$ and $\gamma_2$ be non-separating and separating simple closed curves  on $S_2$. Using Exercises \ref{ex:freq_2} and \ref{ex:main_6} show that on any genus $2$ hyperbolic surface $X \in \mathcal{M}_2$ it is $48$ times more likely for a long random simple closed geodesic to be non-separating rather than separating, i.e., show that
	\[
	\lim_{L \to \infty} \frac{s(X,\gamma_1,L)}{s(X,\gamma_2,L)} = 48.
	\]
\end{exercise}

\section{Beyond simple closed geodesics} 

\subsection*{Outline of this section.} In this section we give a brief overview of several counting results for closed curves on surfaces and other related objects that have been proved since the debut of Mirzakhani's thesis. Rather than exhaustively covering the great amount of material available we aim at giving a landscape picture of the relevance of these results and of the variety of techniques behind their proofs. In particular, although many of the results that follow hold for surfaces with punctures, we focus on the case of closed surfaces. The reader is encouraged to look into the cited references for more details.

\subsection*{General length functions.} For the rest of this section we fix an integer $g \geq 2$ and a connected, oriented, closed surface $S_g$ of genus $g$. Let $\gamma := \smash{\sum_{i=1}^k a_i \gamma_i}$ be a simple closed multi-curve on $S_g$. Recall that $\mathcal{ML}_g$ denotes the space of measured geodesic laminations on $S_g$ and that this space supports a natural $\mathbf{R}_+$ scaling action. Consider a continuous function $\ell \colon \mathcal{ML}_g \to \mathbf{R}_+$ that is homogeneous, i.e., such that $\ell(t \cdot \lambda) = t \cdot \ell(\lambda)$ for every $t >0$ and every $\lambda \in \ml$. Interesting examples of such functions include the extremal length with respect to a given conformal structure \cite{Ker80} and the length of geodesic representatives with respect to an arbitrary negatively curved metric \cite{O90}. For every $L > 0$ consider the counting function
\[
s(\ell,\gamma,L) := \#\{\alpha \in \mcg \cdot \gamma \ | \ \ell(\alpha) \leq L \}.
\]

Recall that $\mu_{\mathrm{Thu}}$ denotes the Thurston measure on $\ml$. Given a continuous, homogeneous function $\ell \colon \ml \to \mathbf{R}_+$ consider the finite, positive constant
\[
B(\ell) := \mu_{\mathrm{Thu}}\left( \left\lbrace \lambda \in \ml \ | \ \ell(\lambda) \leq 1 \right\rbrace\right).
\]
Recall the definition of the constants $c(\gamma) > 0$ and $b_g > 0$ introduced in (\ref{eq:freq}) and (\ref{eq:bg}). A careful consideration of the techniques introduced in the proof of Theorem \ref{theo:main_5} show that the same asymptotic estimates can be proved in this more general setting. Indeed, the following holds.

\begin{exercise}
	Let $\gamma:= \smash{\sum_{i=1}^k a_i \gamma_i}$ be an integral simple closed multi-curve on $S_g$ and $\ell \colon \ml \to \mathbf{R}_+$ be a continuous, homogeneous function. Show that the following identity holds
	\[
	\lim_{L \to \infty} \frac{s(\ell,\gamma,L)}{L^{6g-6}} = \frac{c(\gamma) \cdot B(\ell)}{b_g}.
	\]
	\textit{Hint: Follow the same approach as in Exercise \ref{ex:main_6}.}
\end{exercise}

\subsection*{Non-simple closed geodesics.} One can also consider counting problems for closed geodesics that are not simple. Extending the definition for simple closed curves, we say two closed curves on homeomorphic surfaces have the same topological type if there exists a homeomorphism between the surfaces that identifies the free homotopy classes of the curves.  A closed curve is said to be filling if every homotopically non-trivial closed curve on the surface intersects the curve.

Let $X \in \mathcal{T}_g$ be a marked hyperbolic structure and $\gamma$ be a filling closed curve on $S_g$. For every $L > 0$ consider the counting function
\[
f(X,\gamma,L) := \# \{\alpha \in \mcg \cdot \gamma \ | \ \ell_{\alpha}(X) \leq L\}.
\]
This quantity does not depend on the marking of $X \in \mathcal{T}_g$ and corresponds to the number of closed geodesics on $X$ of the same topological type as $\gamma$ and length $\leq L$. Consider the finite constant
\[
c(\gamma) := \mu_{\mathrm{Thu}}\left(\{\lambda \in \ml \ | \ i(\gamma,\lambda) \leq 1 \} \right) / \#\mathrm{Stab}(\gamma),
\]
where $i(\gamma,\lambda)$ denotes the geometric intersection number between $\gamma$ and $\lambda$ \cite{Bon88} and $\mathrm{Stab}(\gamma) \subseteq \mcg$ denotes the stabilizer of $\gamma$ with respect to the natural mapping class group action. Recall the definition of the Mirzakhani function $B \colon \mathcal{M}_g \to \mathbf{R}_+$ in (\ref{eq:B}). In \cite{Mir16}, Mirzakhani introduced novel techniques, inspired by work of Margulis \cite{Mar04} and relying on previous work of herself on the ergodic theory of the earthquake flow \cite{Mir08a}, to prove asymptotic formulas for the counting function above.

\begin{theorem}
	\label{theo:mir16}
	\cite[Theorem 1.1]{Mir16}
	Let $X \in \mathcal{M}_g$ be a hyperbolic structure and $\gamma$ be a filling closed curve on $S_g$. Then, the following asymptotic formula holds,
	\[
	\lim_{L \to \infty} \frac{f(X,\gamma,L)}{L^{6g-6}} = \frac{c(\gamma) \cdot B(X)}{b_g}.
	\]
\end{theorem}

A result analogous to Theorem \ref{theo:mir16} for closed geodesics of any topological type was later proved by Erlandsson and Souto \cite{ES19} using original arguments introduced in their earlier work \cite{ES16}. One of the main ideas of their work is to study how injective the procedure of smoothing a non-simple closed geodesic at its self-intersections is for closed geodesics of a given topological type.

\subsection*{Geodesic currents.} In \cite{Bon88}, Bonahon gave a unified treatment of several seemingly unrelated notions of length for closed curves on closed, orientable surfaces using the concept of geodesic currents. 

To define geodesic currents let us endow the surface $S_g$ with an auxiliary hyperbolic metric. The projective tangent bundle $PTS_g$ admits a $1$-dimensional foliation by lifts of geodesics on $S_g$. A geodesic current on $S_g$ is a Radon transverse measure of the geodesic foliation of $PTS_g$. Equivalently, a geodesic current on $S_g$ is a $\pi_1(S_g)$-invariant Radon measure on the space of unoriented geodesics of the universal cover of $S_g$. Endow the space of geodesic currents on $S_g$ with the weak-$\star$ topology. Different choices of auxiliary hyperbolic metrics on $S_g$ yield canonically identified spaces of geodesic currents \cite[Fact 1]{Bon88}. Denote the space of geodesic currents on $S_g$ by $\mathcal{C}_g$. This space supports a natural $\mathbf{R}^+$ scaling action and a natural $\mcg$ action \cite[\S 2]{RS19}. 

Free homotopy classes of weighted, unoriented closed curves on $S_g$ embed into $\mathcal{C}_g$ by considering their geodesic representatives with respect to any auxiliary hyperbolic metric. By work of Bonahon \cite[Proposition 2]{Bon88}, this embedding is dense. Moreover, the geometric intersection number pairing for closed curves on $S_g$ extends in a unique way to a continuous, symmetric, bilinear pairing $i(\cdot,\cdot)$ on $\mathcal{C}_g$ \cite[Proposition 3]{Bon88}.  In this sense, geodesic currents are to closed curves what measured geodesic laminations are to simple closed curves.

Many different metrics on $S_g$ embed into $\mathcal{C}_g$ in such a way that the geometric intersection number of the metric with any closed curve is equal to the length of the geodesic representatives of the closed curve with respect to the metric. The geodesic current corresponding to any such metric is usually refered to as its Liouville current. Examples of metrics admitting Liouville currents include hypebolic metrics \cite{Bon88} and negatively curved Riemannian metrics \cite{O90}.

A geodesic current $\alpha \in \mathcal{C}_g$ is said to be filling if $i(\alpha,\beta) > 0$ for every non-zero $\beta \in \mathcal{C}_g$. Relevant examples of filling geodesic currents include free homotopy classes of unoriented filling closed curves and the Liouville currents introduced above. Filling geodesic currents $\alpha \in \mathcal{C}_g$ have finite stabilizers $\mathrm{Stab}(\alpha) \subseteq \mcg$ with respect to the natural mapping class group action.

Consider a continuous function $\ell \colon \mathcal{C}_g \to \mathbf{R}_+$ that is homogeneous, i.e., such that $\ell(t \cdot \alpha) = t \cdot \ell(\alpha)$ for every $t > 0$ and every $\alpha \in \mathcal{C}_g$. Let $\alpha \in \mathcal{C}_g$ be a filling geodesic current. For every $L > 0$ consider the counting function
\[
\mathrm{cur}(\ell,\alpha,L) := \# \{\beta \in \mcg \cdot \alpha \ | \ \ell(\beta) \leq L \}.
\]

In \cite{RS19}, Rafi and Souto introduced novel methods for studying the asymptotics of these counting functions. Given any filling geodesic current $\alpha \in \mathcal{C}_g$ consider the finite positive constant
\begin{equation*}
c(\alpha) := \mu_{\mathrm{Thu}}\left(\{\lambda \in \ml \ | \ i(\alpha,\lambda) \leq 1\}\right) / \#\mathrm{Stab}(\alpha).
\end{equation*}

\begin{theorem}\cite[Main Theorem]{RS19}
	\label{theo:rs}
	Let $\ell \colon \mathcal{C}_g \to \mathbf{R}_+$ be a continuous, homogeneous function and $\alpha \in \mathcal{C}_g$ be a filling geodesic current. Then, the following asymptotic estimate holds,
	\[
	\lim_{L \to \infty} \frac{\mathrm{cur}(\ell,\alpha,L)}{L^{6g-6}} = \frac{B(\ell) \cdot c(\alpha)}{b_g}.
	\]
\end{theorem}

Theorem \ref{theo:rs} was later generalized by Erlandsson and Souto for non-filling geodesic currents \cite{ES20}. As a remarkable application of Theorem \ref{theo:rs}, Rafi and Souto proved an asymptotic formula for the number of points in a mapping class group of Teichmüller space $\mathcal{T}_g$ that lie within a ball of given center and large radius with respect to Thurston's asymmetric metric \cite[Theorem 1.1]{RS19}. Thurston's asymmetric metric quantifies the minimal Lipschitz constant among Lipschitz maps between marked hyperbolic structures on $S_g$ \cite{Thu98}.

\subsection*{Square-tiled surfaces.} A square-tiled surface is a closed, connected, oriented surface constructed from finitely many disjoint unit area squares on the complex plane, with sides parallel to the real and imaginary axes, by identifying pairs of sides by translation and/or $180\degree$ rotation. We assume square-tiled surfaces have no points of cone angle $\pi$. The horizontal core multi-curve of a square tiled-surface is the integrally weighted simple closed multi-curve obtained by concatenating the horizontal segments running through the middle of each square. The vertical core multi-curve of a square tiled-surface is defined in an analogous way. See Figure \ref{fig:example} for an example. 

\begin{figure}[h]
	\centering
	\begin{subfigure}[b]{0.285\textwidth}
		\centering
		\includegraphics[width=0.6\textwidth]{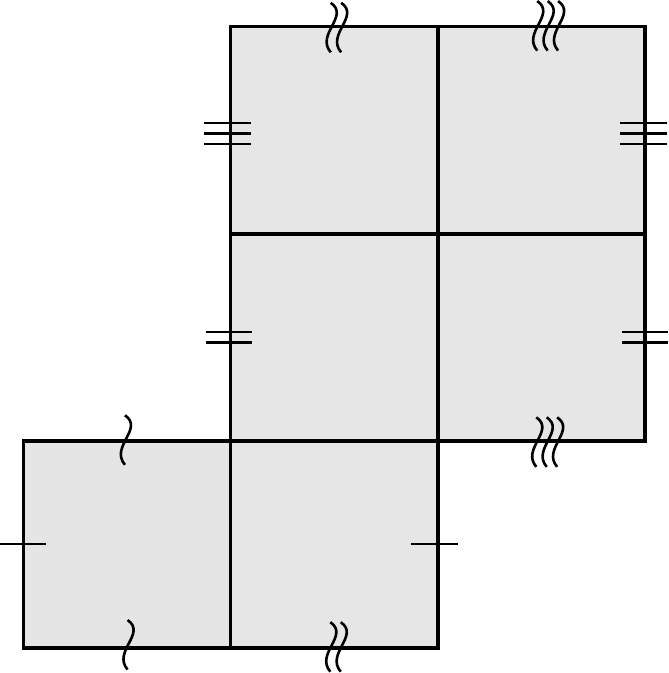}
		\caption{Square-tiled surface.}
		\label{fig:sq_tiled}
	\end{subfigure}
	\quad \quad \quad
	~ 
	\begin{subfigure}[b]{0.31\textwidth}
		\centering
		\includegraphics[width=0.6\textwidth]{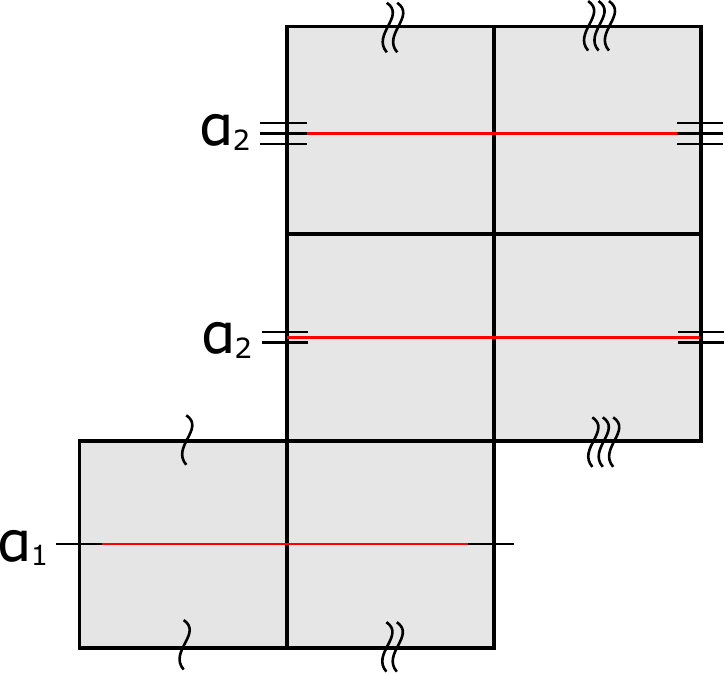}
		\caption{Horizontal core multi-curve.}
		\label{fig:horizontal}
	\end{subfigure}
	\quad \quad \quad
	~ 
	\begin{subfigure}[b]{0.285\textwidth}
		\centering
		\includegraphics[width=0.6\textwidth]{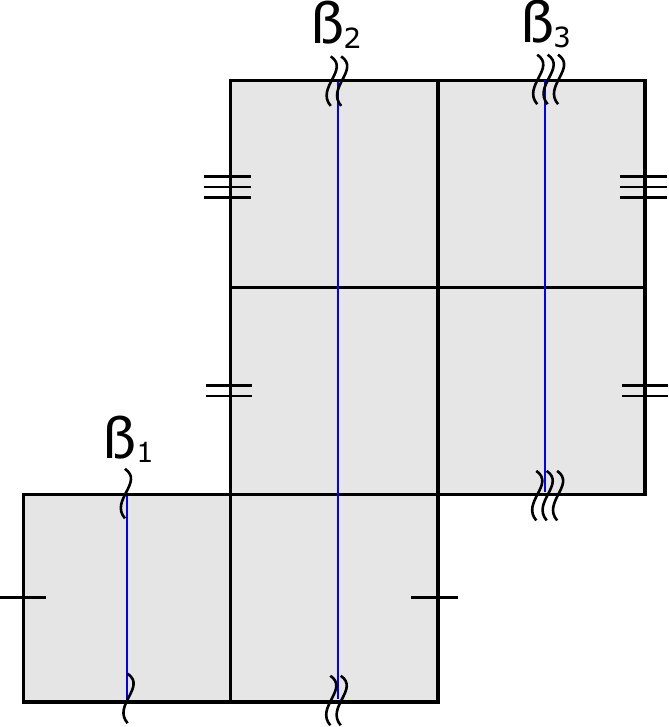}
		\caption{Vertical core multi-curve.}
		\label{fig:vertical}
	\end{subfigure}
	\caption{Example of a square-tiled surface of genus $2$. The horizontal core multi-curve is $\alpha_1 + 2 \alpha_2$. The vertical core multi-curve is $\beta_1 + \beta_2 + \beta_3$.} \label{fig:example}
\end{figure}

Recall that two integrally weighted simple closed multi-curves on homeomorphic surfaces have the same topological type if there exists a homeomorphism between the surfaces mapping one multi-curve to the other preserving the weights. Let $\alpha := \smash{\sum_{i=1}^k a_i \alpha_i}$ be an integral simple closed multi-curve on $S_g$. For every $L > 0$ consider the counting function
\[
sq(\alpha,L) := \#
\left\lbrace 
\begin{array}{c}
\text{square-tiled surfaces with horizontal core multi-curve}\\
\text{of the same topological type as $\alpha$ and $\leq L$ squares} 
\end{array}
\right\rbrace /\sim,
\]
where $\sim$ denotes the equivalence relation induced by cut and paste operations. As in the counting problems above, we are interested in the asymptotics of $\mathrm{sq}(\alpha,L)$ as $L \to \infty$. Denote by
\begin{equation}
	\label{eq:epsilon}
	\epsilon_{g} := \left\lbrace
	\begin{array}{ccl}
		2 & \text{if} & g=2, \\ 
		1 & \text{if} & g \neq 2,\\
	\end{array} \right.
\end{equation}
the number of automorphisms of a generic square-tiled surface of genus $g$. Recall the definition of the constants $c(\gamma) > 0$ introduced in (\ref{eq:freq}).

\begin{theorem}
	\label{theo:square_tiled}
	Let $\alpha := \smash{\sum_{i=1}^k a_i \alpha_i}$ be an integral simple closed multi-curve on $S_g$. Then, 
	\[
	\lim_{L \to \infty} \frac{sq(\alpha,L)}{L^{6g-6}} = \frac{\epsilon_{g} \cdot c(\alpha)}{2^{2g-3}}.
	\]
\end{theorem}

One can also consider more refined counting functions of square-tiled surfaces. Let $\alpha:= \smash{\sum_{i=1}^k \alpha_i \alpha_i}$ and $\beta := \smash{\sum_{j=1}^l b_j \beta_j}$ be integral simple closed multi-curves on $S_g$. For every $L > 0$ denote
\[
sq(\alpha,\beta,L) := \#
\left\lbrace 
\begin{array}{c}
\text{square-tiled surfaces with horizontal core multi-curve of }\\
\text{the same topological type as $\alpha$ and  vertical core multi-curve}\\
\text{ of the same topological type as $\beta$ and $\leq L$ squares} 
\end{array}
\right\rbrace /\sim,
\]
where $\sim$ denotes the equivalence relation induced by cut and paste operations.

\begin{theorem}
	\label{theo:square_tiled_2}
	Let $\alpha := \smash{\sum_{i=1}^k a_i \alpha_i}$ and $\beta := \smash{\sum_{j=1}^l b_j \beta_j}$ be integral simple closed multi-curves on $S_g$. Then, the following asymptotic formula holds,
	\[
	\lim_{L \to \infty} \frac{sq(\alpha,\beta,L)}{L^{6g-6}} = \frac{\epsilon_{g} \cdot c(\alpha) \cdot c(\beta)}{2^{2g-3}}.
	\]
\end{theorem}

Theorems \ref{theo:square_tiled} and \ref{theo:square_tiled_2} were originally proved by Delecroix, Goujard, Zograf, and Zorich using algebro-geometric methods \cite{DGZZ19}. Different proofs of these theorems were later provided in \cite{Ara19a}. These proofs made crucial use of the results of Mirzakhani discussed in this survey.

\subsection*{Individual components of multi-curves.} An ordered simple closed multi-curve on $S_g$ is a tuple $\gamma := (\gamma_1,\dots,\gamma_k)$ with $1 \leq k \leq 3g-3$ of pairwise non-isotopic and non-intersecting simple closed curves. Two ordered simple closed multi-curves on homeomorphic surfaces have the same topological type if there exists a homeomorphism between the surfaces mapping one multi-curve to the other respecting the orders. Multi-geodesics on hyperbolic surfaces are multi-curves all of whose components are geodesics. Inspired by Mirzakhani's simple close geodesic counting theorems, Wolpert conjectured that analogous results should hold for countings of simple closed multi-geodesics that keep track of the hyperbolic length of individual components, rather than just the total hyperbolic length. 

For instance, let $X \in \mathcal{M}_g$ be a closed, connected, oriented hyperbolic surface of genus $g \geq 2$ and $\gamma := (\gamma_1,\dots,\gamma_k)$ be an ordered simple closed multi-curve on $S_g$ with $1 \leq k \leq 3g-3$ components. For every $L > 0$ consider the counting function 
\begin{align}
m(X,\gamma,L) \label{eq:max_count}
:= \# \left\lbrace
\begin{array}{c}
\text{ordered simple closed multi-geodesics $\alpha :=(\alpha_1,\dots,\alpha_k)$ on $X$} \\ \text{of the same topological type as $\gamma$ with $\max_{i=1,\dots,k} \ell_{\alpha_i}(X) \leq L$} \hspace{-0.1cm} \nonumber
\end{array}
\right\rbrace.
\end{align}

\begin{exercise}
	\label{ex:freq_max}
	Let $\gamma := (\gamma_1,\dots,\gamma_k)$ be an ordered simple closed multi-curve on $S_g$. Using Mirzakhani's integration formulas give an explicit expression for
	\begin{equation*}
	\label{eq:Pmax}
	M(\gamma,L) := \int_{\mathcal{M}_g} m(X,\gamma,L) \thinspace d\widehat{\mu}_\mathrm{wp}(X) 
	\end{equation*}
	in terms of Weil-Petersson volume polynomials and conclude that $M(\gamma,L)$ is a polynomial in $L$ of degree $6g-6$. \textit{Hint: Follow the same approach as in Exercise \ref{ex:freq}.}
\end{exercise}

Following Exercise \ref{ex:freq_max} we define 
\[
m(\gamma) := \lim_{L \to \infty} \frac{M(\gamma,L)}{L^{6g-6}}.
\]

\begin{theorem}
	\label{theo:count_max}
	Let $X \in \mathcal{M}_g$ be a hyperbolic structure on $S_g$ and $\gamma := (\gamma_1,\dots,\gamma_k)$ with $1 \leq k \leq 3g-3$ be an ordered simple closed multi-curve on $S_g$. Then, the following asymptotic formula holds,
	\[
	\lim_{L \to \infty} \frac{m(X,\gamma,L)}{L^{6g-6}} = \frac{m(\gamma) \cdot B(X)}{b_g}.
	\]
\end{theorem}

Theorem \ref{theo:count_max} was proved independently by the author \cite{Ara20a} and Liu \cite{Liu19} using general averaging and unfolding methods introduced by Margulis in his thesis \cite{Mar04}. In  \cite{Ara20a} a generalization of Theorem \ref{theo:count_max} for ordered filling closed multi-curves was proved using techniques introduced by Mirzakhani in \cite{Mir16}. A result analogous to Theorem \ref{theo:count_max} for all topological types of closed curves has since been proved by Erlandsson and Souto using original methods \cite{ES20}.

\subsection*{Effective results.} The counting results discussed above only provide asymptotic estimates without explicit error terms. The search for effective estimates for counting problems of simple closed geodesics has only seen progress in recent years. In \cite{EMM19}, Eskin, Mirzakhani, and Mohammadi introduced new methods fundamentally based on Teichmüller dynamics to prove the following effective version of Mirzakhani's simple closed geodesic counting theorem.

\begin{theorem}
	\cite{EMM19}
	\label{theo:EMM}
	Let $X \in \mathcal{M}_g$ be a hyperbolic structure on $S_g$ and $\gamma := \sum_{i=1}^k a_i \gamma_i$ be an integral simple closed multi-curve on $S_g$. Then, for every $L > 0$,
	\[
	s(X,\gamma,L) = \frac{c(\gamma) \cdot B(X)}{b_g} \cdot L^{6g-6} + O_{X,\gamma}\left(L^{6g-6-\kappa}\right),
	\]
	where $\kappa = \kappa(g) > 0$ is a positive constant depending only on the genus $g \geq 2$.
\end{theorem}

Even more recently, in \cite{Ara21a}, novel methods were introduced by the author to prove analogous effective estimates for countings of filling closed geodesics of a given topological type.  These methods rely on recent progress made in the prequels \cite{Ara20b} and \cite{Ara20c} on the study of the effective dynamics of the mapping class group on Teichmüller space and the space of closed curves of a closed, orientable surface. These recent developments in turn rely on the exponential mixing rate, the hyperbolicity, and the renormalization dynamics of the Teichmüller geodesic flow as their main driving forces.

\begin{theorem}
	\cite{Ara21a}
	\label{theo:me}
	Let $X \in \mathcal{M}_g$ be a hyperbolic structure on $S_g$ and $\gamma$ be a filling closed curve on $S_g$. Then, for every $L > 0$,
	\[
	f(X,\gamma,L) = \frac{c(\gamma) \cdot B(X)}{b_g} \cdot L^{6g-6} + O_{X,\gamma}\left(L^{6g-6-\kappa}\right),
	\]
	where $\kappa = \kappa(g) > 0$ is a positive constant depending only on the genus $g \geq 2$.
\end{theorem}

The methods introduced in \cite{Ara21a} can also be used to prove an effective version of Rafi's and Souto's asymptotic estimate for the number of points in a mapping class group orbit of Teichmüller space that lie within a ball of given center and large radius with respect to Thurston's asymmetric metric. An analogous result for balls in the Teichmüller metric was proved by the author in \cite{Ara20b} building on previous work of Athreya, Bufetov, Eskin, and Mirzakhani \cite{ABEM12}.

\bibliographystyle{amsalpha}


\bibliography{bibliography}

\end{document}